\documentclass[11pt]{amsart}

\usepackage[utf8]{inputenc}
\usepackage{parskip}
\usepackage{todonotes}
\usepackage{caption}

\usepackage{graphicx}
\usepackage{xxcolor}
\usepackage{color}
\usepackage{comment}

\usepackage{appendix}

\usepackage{amsmath}
\usepackage{amssymb}
\usepackage{amsthm}
\usepackage{scrextend}
\usepackage{bbm}
\usepackage{stmaryrd}
\usepackage{geometry}
\usepackage{physics}
\usepackage{enumerate}
\usepackage{tikz, pgffor}
\usepackage[linktocpage=true,colorlinks=true,linkcolor=blue,citecolor=magenta]{hyperref}
\usepackage{cleveref}
\usepackage{mathtools}
\usepackage{amsfonts}
\usepackage{fakeMnSymbol}

\usepackage{esint} 

\newcommand{\eps}{\varepsilon}
\newcommand{\sdist}{\mathbf{s}}
\renewcommand{\{}{\left\lbrace}
  \renewcommand{\}}{\right\rbrace}
\let\div\relax
\DeclareMathOperator*{\interior}{int}
\DeclareMathOperator{\div}{div}
\DeclareMathOperator{\id}{id}
\DeclareMathOperator{\Id}{Id}
\DeclareMathOperator*{\dist}{dist}
\DeclareMathOperator*{\diam}{diam}
\DeclareMathOperator{\supp}{supp}

\allowdisplaybreaks 

\theoremstyle{plain}
\newtheorem{theorem}{Theorem}

\newtheorem{lemma}{Lemma}
\newtheorem{corollary}{Corollary}
\newtheorem{proposition}{Proposition}

\theoremstyle{definition}
\newtheorem{definition}{Definition}
\newtheorem{step}{Step}

\theoremstyle{remark}
\newtheorem{remark}{Remark}

\title[Uniqueness and stability for surface diffusion]{A uniqueness and stability principle for\\ surface diffusion}

\author{Milan Kroemer and Tim Laux}

\address{Hausdorff Center for Mathematics and Institute for Applied Mathematics, University of Bonn, Villa Maria, Endenicher Allee 62, 53115 Bonn, Germany}

\email{ \nolinkurl{{ milan.kroemer, tim.laux}@hcm.uni-bonn.de} }

\begin{document}

\maketitle

\begin{abstract}
	We derive a uniqueness and stability principle for surface diffusion before the onset of singularities. 
	The perturbations, however, are allowed to undergo topological changes. 
	The main ingredient is a relative energy inequality, which in turn relies on the explicit construction of (volume-preserving) gradient flow calibrations. 
	The proof applies to stationary solutions in any dimension and to general
  smooth solutions in two dimensions.

  \medskip
	
	
  \medskip
	
	\noindent \textbf{Mathematical Subject Classification (MSC 2020)}: 
  53E40 (primary); 35B35; 53A15; 53E10 (secondary)
\end{abstract}


\section{Introduction}
\label{sec:intro}

Surface diffusion is the most fundamental continuum model describing  the evolution of solid material surfaces.
The equation was first introduced by Mullins~\cite{Mullins1957} to describe
the phenomenon of thermal grooving at grain boundaries of heated polycrystals.
Surface diffusion describes the slow change of the shape of a solid body caused by the diffusion of particles along the surface to energetically favorable locations. 
It arises naturally from the conservation law relating the normal velocity $V=V(x,t)$ of the boundary $\Sigma(t)$ of the solid body and the flux of particles $j=j(x,t)$ moving along $\Sigma(t)\subset\mathbb{R}^d$ via
\begin{align*}
  V(\cdot,t)+ \div_{\Sigma(t)}  j(\cdot,t) =0 \quad \text{on } \Sigma(t)
\end{align*}
together with the Nernst--Planck relation stipulating that the average flux of atoms on the material surface follows the negative (mean) curvature gradient,  i.e., 
\begin{align*}
  j(\cdot,t)= -\nabla_{\Sigma(t)} H(\cdot,t) \quad \text{on } \Sigma(t).
\end{align*}
These two equations then precisely yield the surface diffusion equation
\begin{equation}
  \label{eq:surface-diffusion-eq}
  V(\cdot,t)=\Delta_{\Sigma(t)} H(\cdot,t) \quad \text{on } \Sigma(t).
\end{equation}

While surface diffusion is crucial for applications, the mathematical literature is comparably thin.  
This is mainly due to the fragile structure of the equation and the intricate behavior of solutions.
From the viewpoint of partial differential equations, surface diffusion is a quasilinear degenerate parabolic fourth-order equation.  The degeneracy is due to the geometric invariance of the equation just as in its second-order analog, the mean curvature flow. 
The evolution~\eqref{eq:surface-diffusion-eq} has a regularizing effect in the
sense that edges and corners are smoothed out immediately.
As in many geometric evolution equations, despite the immediate smoothing
effect, one has to expect singularity formation in finite time.
The regularizing effect, the fact that surface diffusion can lose convexity and the formation of singularities can be observed in numerical simulations, see the overview article by Garcke~\cite[Fig.\ 3]{Garcke2013}.
Using the parabolic structure, Escher, Mayer and Simonett~\cite{Escher1998} have shown short time existence and uniqueness of classical parametrized solutions for sufficiently regular initial data in arbitrary dimensions.
Further they give numerical examples of embedded curves evolving by surface diffusion that develop singularities in finite time.
Giga and Ito~\cite{Giga1998} have shown that there exist unique local solutions for immersed $H^4$-initial curves.
Chou~\cite{Chou2003} gives a sharp criterion for finite time blowup.
Therefore, it is in principle interesting to consider weak solutions.
However, to the best of our knowledge, no such theory has been developed.

Spheres play a crucial role in surface diffusion since any disjoint union of spheres is a stationary solution. 
They are asymptotically stable in the following sense. 
If the initial datum is sufficiently close to a sphere, then there exists a global solution and it converges to a sphere,	see~\cite{Elliott1997} for the 2-dimensional case and~\cite{Escher1998} for higher dimensions.
Wheeler~\cite{Wheeler2013,Wheeler2022} shows that the evolution of $H^2$-perturbations of a circle converge exponentially fast to a circle in the long-term limit.
A similar result holds in higher dimensions~\cite{Wheeler2012}.
Miura and Okabe~\cite{Miura2020} study the analogous problem for immersed curves.
There are also further interesting stationary solutions of surface diffusion.
For the standard double bubble, Abels, Arab, and Garcke~\cite{Abels2021} showed the stability in the planar case.
Garcke and G\"oswein~\cite{Garcke2021} extended this result to higher dimensions.
For Delaunay surfaces, Kohsaka~\cite{KohY:2017} gives a sufficient condition for stability and in-stability under surface diffusion.
In more generality, Acerbi et al.~\cite{Acerbi2019} show that any three-dimensional configuration which is periodic and strictly stable for the area functional is exponentially stable for the surface diffusion flow.

In contrast to the mean curvature flow, surface diffusion does not satisfy any
comparison principle so that viscosity solution techniques and simple
geometric comparisons fail.
Initially embedded curves can evolve to self intersections, see~\cite{Mayer2000} and~\cite{Blatt2010}.
Even if the initial surface is the graph of a function, it can lose this property under the evolution~\cite{Elliott2001}.
Giga and Ito~\cite{Giga1999} show that there exist simple closed curves which
lose convexity while remaining simple before developing singularities, see
also~\cite{Blatt2010} for some more recent results.

For the present work, the key structural property of surface diffusion is its gradient flow structure: any (sufficiently regular) solution of~\eqref{eq:surface-diffusion-eq} satisfies the energy dissipation relation
\begin{align}\label{eq:intro_ddtArea}
  \frac{d}{dt} \operatorname{Area}(\Sigma(t)) = \int_{\Sigma(t)} V(\cdot,t)H(\cdot,t)\, dS = -\int_{\Sigma(t)} |\nabla_{\Sigma(t)} H(\cdot,t)|^2 \, dS\leq 0.
\end{align}
In other words, the evolution~\eqref{eq:surface-diffusion-eq} follows the steepest descent in an energy landscape;
the energy is the surface area while the metric tensor is the $H^{-1}$ product on normal velocities. 
This structure is the crucial ingredient for the present work and will allow us
to derive a stability and uniqueness result.
Our main results, Theorem~\ref{thm:stability-stationary-sol} and
Theorem~\ref{thm:main-thm}, show that stationary points are stable in any
dimension and that smooth solutions are stable in two dimensions.
The perturbations are allowed to be rather wild, e.g., we can add many small bubbles.

The basic idea is to monitor the evolution of a relative energy modeled after the area functional appearing in~\eqref{eq:intro_ddtArea}. 
This general strategy was introduced recently for multiphase mean curvature flow by Fischer, Hensel, Simon and one of the authors~\cite{FHLS}. 
The key difficulty in that work is then to construct so-called gradient flow calibrations.
This framework was extended to the volume-preserving mean curvature flow in~\cite{Laux2022}, where the notion of volume-preserving calibrations was introduced. 
While in volume-preserving mean curvature flow, the volume is preserved by an artificial Lagrange-multiplier, surface diffusion preserves the volume of the enclosed region~$\Omega(t)$ naturally, as can be seen at the simple computation
\begin{align*}
  \frac{d}{dt} \operatorname{Vol}(\Omega(t)) 
  = \int_{\Sigma(t)} V(\cdot,t) \, dS 
  =\int_{\Sigma(t)} \Delta_{\Sigma(t)} H(\cdot,t) \, dS 
  =0.
\end{align*}
We will construct a variant of the volume-preserving calibrations
from~\cite{Laux2022} in our setting of surface diffusion. 
The main additional challenge in our proof is that the metric structure in the
gradient flow is not the standard $L^2$ scalar product on normal velocities
but the $H^{-1}$ product on the evolving surfaces.
That means that we have to introduce certain potentials on the moving surfaces
in order to capture the precise energy-dissipation mechanism.
The crucial step in our argument is to compare the velocity potentials of two
surfaces, which in particular requires replacing covariant derivatives on one
surface by extensions of the corresponding derivatives on another one.

There has been continuous interest from the applied mathematics community in
surface diffusion because it arises in many physical models.
Cahn, Elliott and Novick-Cohen~\cite{CEN} introduced a Cahn--Hilliard equation
with degenerate mobility and performed formal matched asymptotics indicating
that the sharp-interface limit is governed by surface diffusion.
However, as explained in the review~\cite{GSK} by Gugenberg, Spatschek,
Kassner, different potentials can lead to diffusion in the bulk:
There is a fine interplay between potential and mobility that restricts the
diffusion only to the diffuse interface layer.
R\"atz, Ribalta and Voigt~\cite{Raetz2006} introduced a similar Cahn--Hilliard
equation which allows for standard double-well potentials.
However, this model does not carry a variational structure.
Most recently, another modified Cahn--Hilliard model has been proposed by
Bretin et al.~\cite{Bretin2022} which can be written as the weighted $H^{-1}$
gradient flow of the Cahn--Hilliard energy.
Again, formal matched asymptotics suggest that in the sharp-interface limit
this yields surface diffusion.
Independently, Salvalaglio et al.~\cite{Salvalaglio2021} defined a similar
modified  Cahn--Hilliard model which has a gradient flow structure
\begin{align*}
  \eps \partial_t u=\nabla\cdot(M(u)\nabla\delta\! E_\eps(u)),
\end{align*}
where $E_\eps$ is the modified Cahn--Hilliard energy functional
\begin{align*}
  E_\eps(u)=\int g_0(u)\left( \frac{\eps}{2}|\nabla u|^2+\frac{1}{\eps}W(u) \right)\,dx,
\end{align*}
and both the mobility function $M$ and the inverse
restriction function $1/{g_0}$ vanish at the two stable configurations of the potential $W$.
One of the main interests in studying such phase-field models is that they
give rise to efficient numerical algorithms to approximate surface diffusion.
Indeed, a direct implementation of finite elements for parametrizations, as
done by B\"ansch et al.~\cite{Baensch2005} cannot resolve singularities.
Phase-fields on the other hand give a natural choice of evolution through singularities.
The numerical implementation of phase-fields is commonly done by a
convex-concave splitting in the spirit of Eyre~\cite{Eyre1998} because this
leads to an unconditionally stable scheme.
For the model introduced in~\cite{Raetz2006}, this was performed by Backofen et al.~\cite{Backofen2019}.

In principle, uniqueness and stability proofs are interesting for these
applications because they have the potential to be lifted to approximation
schemes such as phase-field models.
For mean curvature flow and its Allen--Cahn approximation, this was carried
out recently by Fischer, Simon and one of the authors~\cite{FLS}.
We are confident that some of the ideas developed here will be useful for
studying the sharp-interface limit of the Cahn--Hilliard equations developed
in~\cite{Bretin2022,Salvalaglio2021}.
In principle, another possible convergence approach would be via compactness
and weak solutions as was done by the authors in~\cite{KL} for the
Cahn--Hilliard equation with disparate mobilities, i.e., a mobility function
that vanishes in only one of the two stable phases.

\section{Definitions and notation}
\label{sec:definitions-and-notation}

We begin by defining strong and weak solution to surface diffusion.

\begin{definition}[Smooth solution to surface diffusion]\label{def:strong}
  Let $\Sigma^*=(\Sigma^*(t))_{t\in[0,T^*]}$ be a family of closed surfaces
  $\Sigma^*(t)=\partial\Omega^*(t)$ with $\Omega^*(t)\subset\mathbb{R}^d$ open and bounded.
  We say that $\Sigma^*$ is a smooth solution to surface diffusion flow,
  if $\Sigma^*(t)$ is $C^{6,\alpha}$, for some $\alpha\in(0,1)$, and
  the normal velocity $V^*\in C^{2,\alpha}$ and mean curvature $H^*$ satisfy
  \begin{align*}
    V^*(\cdot,t)=\Delta_{\Sigma^*(t)}H^*(\cdot,t)\qquad\text{on $\Sigma^*(t)$}.
  \end{align*}
  We call a surface $\Sigma^*$ a stationary solution if it gives rise to a
  constant-in-time solution with $V^*=0$.
\end{definition}

Here $\Delta_{\Sigma^*(t)}$ denotes the Laplace--Beltrami operator on
$\Sigma^*(t)$.

Before we define weak solutions, we define some quantities on finite perimeter
sets which we will need.
Let $\Omega\subset\mathbb{R}^d$ be a set of finite perimeter and let
$\Sigma\coloneqq\partial^*\Omega$.
Further let $\nu\coloneqq-\frac{\nabla\chi_\Omega}{|\nabla\chi_\Omega|}$ denote
the outward unit normal.

Suppose there exists a function $H:\Sigma\rightarrow\mathbb{R}$ such that for
all vector fields $B\in C_c^1(\mathbb{R}^d;\mathbb{R}^d)$ it holds
\begin{equation}\label{eq:def-mean-curvature-weak}
  \int_\Sigma\div_\Sigma B\,d\mathcal{H}^{d-1}
  =\int_\Sigma H\nu\cdot B\,d\mathcal{H}^{d-1},\qquad
  \div_\Sigma B\coloneqq\nabla\cdot B-\nu\cdot\nabla B\nu,
\end{equation}
that is, $H$ is the distributional mean curvature of $\Sigma$.

We say that a function $\varphi:\Sigma\rightarrow\mathbb{R}$ is weakly
differentiable if there exists a vector field
$\nabla_\Sigma\varphi:\Sigma\rightarrow\mathbb{R}^d$ such that for
all $B\in C_c^1(\mathbb{R}^d;\mathbb{R}^d)$ it holds
\begin{align*}
  \int_\Sigma B\cdot\nabla_\Sigma\varphi\,d\mathcal{H}^{d-1}
  &=\int_\Sigma H\nu\cdot(\varphi B)\,d\mathcal{H}^{d-1}
    -\int_\Sigma\varphi\div_\Sigma B\,d\mathcal{H}^{d-1}.
\end{align*}
Note that if $\varphi\in C^1(\mathbb{R}^d)$ is defined on all of $\mathbb{R}^d$,
then the definition of $\nabla_\Sigma\varphi$ coincides with the standard
tangential gradient
\begin{align*}
  \nabla_\Sigma\varphi=(\Id-\nu\otimes\nu)\nabla\varphi.
\end{align*}
Now we define the Sobolev space $H^1(\Sigma)$ by
\begin{align*}
  H^1(\Sigma)\coloneqq\{\varphi\in L^2(\Sigma):\nabla_\Sigma\varphi\in L^2(\Sigma;\mathbb{R}^d)\}.
\end{align*}

\begin{remark}\label{rem:H1-is-hilbert-space}
  The space $H^1(\Sigma)$ is the closure of
  $C_c^\infty(\mathbb{R}^d)$
  under the norm
  \begin{align*}
    \|u\|_{H^1(\Sigma)}^2\coloneqq\|u\|_{L^2}^2+\|\nabla_\Sigma u\|_{L^2}^2.
  \end{align*}
  Indeed, let $\varphi_\eps$ be a mollifier, let $f\in H^1(\Sigma)$
  and let $\mu=\mathcal{H}^1\MNSlefthalfcup\Sigma$.
  Then the function
  \[f_\eps(x)\coloneqq\int_{\mathbb{R}^2}f(y)\varphi_\eps(x-y)\,d\mu(y)\]
  and its gradient $\nabla_\Sigma f_\eps$ converges to $f$, respectively $\nabla_\Sigma f$,
  uniformly on compact sets and thus in $L^2(\Sigma,\mu)$.
  In particular $H^1(\Sigma)$ is a Hilbert space and
  $\nabla_\Sigma\varphi$ satisfies
  \begin{align*}
    \frac{1}{2}\|\nabla_\Sigma\varphi\|_{L^2}^2
    &=\frac{1}{2}\int_\Sigma|\nabla_\Sigma\varphi|^2\,d\mathcal{H}^{d-1}
    =\sup_{\xi}\{\int_\Sigma\nabla_\Sigma\varphi\cdot\xi\,d\mathcal{H}^{d-1}
    -\frac{1}{2}\int_{\Sigma}|\xi|^2\,d\mathcal{H}^{d-1}\},
  \end{align*}
  where we take the supremum over all test vector fields
  $\xi\in C_c^1(\mathbb{R}^d;\mathbb{R}^d)$.
\end{remark}

We define the space $H^{-1}(\Sigma)$ as follows:
Define an operator $\Delta_{\Sigma}$ on $H^1(\Sigma)$ characterized by
\begin{align*}
  -\int_\Sigma(\Delta_\Sigma u)v\,d\mathcal{H}^{d-1}
  &=\int_\Sigma\nabla_\Sigma u\cdot\nabla_\Sigma v\,d\mathcal{H}^{d-1}
  \quad\text{for all $v\in C_c^\infty(\mathbb{R}^d)$.}
\end{align*}
Then we set $H^{-1}(\Sigma)\coloneqq\{\Delta_\Sigma u:u\in H^1(\Sigma)\}$.

Now we have everything we need to define weak solutions to surface diffusion.

\begin{definition}[Weak solution to surface diffusion]\label{def:weak}
  Let $\{\Omega(t)\}_{t\in[0,T)}$ be a family of sets of finite perimeter and
  let $\Sigma=(\Sigma(t)\coloneqq\partial\Omega(t))_{t\in[0,T)}$.
  We say that $\Sigma$ is a weak solution to surface diffusion
  if there exist functions $V(\cdot,t)\in H^{-1}(\Sigma(t))$ and
  $H(\cdot,t)\in\dot{H}^1(\Sigma(t)),\varphi_V(\cdot,t) \in H^1(\Sigma(t))$ such that 
  \begin{enumerate}[(i)]
  \item The function $H(\cdot,t)$ is the mean curvature of the surface
    $\Sigma(t)$ in the sense that for all
    $B\in C^1(\mathbb{R}^d\times(0,T);\mathbb{R}^d)$
    \begin{equation}
      \label{eq:weak-formulation-kappa-curvature}
      \int_0^T\int_{\Sigma(t)}\div_{\Sigma(t)}B\,d\mathcal{H}^{d-1}\,dt
      =\int_0^T\int_{\Sigma(t)}H\nu\cdot B\,d\mathcal{H}^{d-1}\,dt.
    \end{equation}
  \item The function $V(\cdot,t)$ is the normal velocity of the surface
    $\Sigma(t)$ in the sense that for all
    $\zeta\in C^1(\mathbb{R}^d\times[0,T])$ and almost all $T'\in (0,T)$
    \begin{equation}
      \label{eq:def-V-weak}
      \int_{\Omega(T')}\zeta(x,T')\,dx-\int_{\Omega(0)}\zeta(x,0)\,dx
      =\int_0^{T'}\left( \int_{\Omega(t)}\partial_t\zeta(x,t)\,dx
        +\int_{\Sigma(t)}V\zeta\,d\mathcal{H}^{d-1}
      \right)\,dt.
    \end{equation}
  \item The function $\varphi_V$ is the zero-average potential of $V$ in the
    sense that
    $\int_{\Sigma(t)} \varphi_V\,d\mathcal{H}^{d-1} =0$ and for all
    $g\in C_c^\infty(\mathbb{R}^d)$
    \begin{equation}
      \label{eq:phi-V-weak-formulation}
      -\int_0^T\int_{\Sigma(t)}\nabla_{\Sigma(t)} g\cdot\nabla_{\Sigma(t)}\varphi_V\,d\mathcal{H}^{d-1}\,dt
      =\int_0^T\int_{\Sigma(t)}gV\,d\mathcal{H}^{d-1}\,dt.
    \end{equation}
  \item For almost every $T' \in (0,T)$, the following sharp energy dissipation relation is satisfied
    \begin{equation}
      \label{eq:area-decay-assumption}
      \begin{split}
        &\mathcal{H}^{d-1}(\Sigma(T'))-\mathcal{H}^{d-1}(\Sigma(0)) \\
        \leq\,\,&
                  -\frac{1}{2}\int_0^{T'}\int_{\Sigma(t)}|\nabla_{\Sigma(t)}H|^2\,d\mathcal{H}^{d-1}\,dt
                  -\frac{1}{2}\int_0^{T'}\int_{\Sigma(t)}|\nabla_{\Sigma(t)}\varphi_V|^2\,d\mathcal{H}^{d-1}\,dt.
      \end{split}
    \end{equation}
  \item There exists a constant $C<\infty$ such that for a.e.\ $x,t$ we have
    \begin{equation}\label{eq:bound-balls}
      \sup_{r>0}\frac{1}{2r}\mathcal{H}(B^{\Sigma(t)}(x,r))<C.
    \end{equation}
  \end{enumerate}
\end{definition}

\begin{remark}
  Note that, since $H$ is differentiable almost everywhere on $\Sigma$,
  a.e.\ $x\in\Sigma$ is a Lebesgue point of $H$.
  Hence, by the Poincar\'{e} inequality~\eqref{eq:poincare-ineq}, we have $H\in H^1(\Sigma)$.
\end{remark}

We are going to prove a weak-strong uniqueness result.
Before we do this, we introduce some notation.
Suppose $\Sigma^*(t)=\bigcup_{i=1}^{k^*(t)}\Sigma^*_i(t)=\bigcup_{i=1}^{k^*(t)}\partial^*\Omega_i^*(t)$
is a smooth solution to surface diffusion flow with $k^*(t)$ path components.
Since all $\Sigma^*(t)$ are $C^2$ hypersurfaces, no topological changes can occur, so $k^*(t)=k^*$ is constant.

Let $\zeta$ be a cutoff function such that $\zeta(\tilde s)=1-\tilde s^2$ for
$|\tilde s|\leq\delta/2$ and $\zeta(\tilde s)=0$ for $|\tilde s|\geq\delta$.
We need a smallness assumption on $\delta$, which we will state below.

First, we need the following definition.
Let $\Omega^*_i(t)$ be the region enclosed by $\Sigma^*_i(t)$ and define the
signed distance of $\Omega^*(t)=\bigcup_{i=1}^{k^*}\Omega^*_i(t)$ by
\begin{align*}
  \sdist(x,t)&\coloneqq\dist(x,\Omega^*(t))-\dist(x,(\Omega^*(t))^c).
\end{align*}
Then $\sdist(\cdot,t)$ is a smooth function in a neighborhood of $\Sigma^*$.
Therefore, for each $i$ and each $x\in\Sigma^*_i(t)$, there exists $\eps_x>0$ sufficiently
small such that $\nabla\sdist(x)\neq 0$ in $B_{\eps_x}(x)$.
Then the orthogonal projection onto $\Sigma^*_i(t)$ is well defined in the
neighborhood $B_{\eps_x}(x)$ and given by
$\id-\sdist(\cdot,t)\nabla\sdist(\cdot,t)$.
Define
\begin{align*}
  r(\Sigma^*_i(t))\coloneqq\min_{x\in\Sigma^*_i(t)}\eps_x.
\end{align*}
Then $r(\Sigma^*_i(t))>0$ because $\Sigma^*_i(t)$ is compact.
Now our assumption on $\delta$ is
\begin{eqnarray}
  \label{eq:assumption-delta}
  \delta<\min_{0\leq t\leq T^*}\min_{1\leq i\leq k^*}\{
  \min_{j\neq i}\{\dist(\Sigma^*_i(t),\Sigma^*_j(t))/4\},
  r(\Sigma^*_i(t))/4
  \},
\end{eqnarray}
and, for each $t\in[0,T^*]$ we define a tubular neighborhood of $\Sigma^*(t)$ by
\begin{align*}
  \mathcal{U}_\delta(t)\coloneqq\{x\in\mathbb{R}^d:\dist(x,\Sigma^*(t))<\delta\}.
\end{align*}
Define $\xi:\mathbb{R}^d\times[0,T^*]\rightarrow\mathbb{R}^d$ by
\begin{equation}
  \label{eq:def-xi}
  \xi(x,t)\coloneqq\zeta(\sdist(x,t))\nabla\sdist(x,t),
\end{equation}
We define the orthogonal projection
$\pi^*(\cdot,t):\mathcal{U}_\delta(t)\rightarrow\Sigma^*(t)$ onto $\Sigma^*(t)$:
\begin{equation}
  \label{eq:def-orthogonal-proj-sigma-star}
  \pi^*(\cdot,t)
  \coloneqq\id-\sdist(\cdot,t)\nabla\sdist(\cdot,t).
\end{equation}
The assumption~\eqref{eq:assumption-delta} on $\delta$ guarantees that $\pi^*(\cdot,t)$
is well defined for $t\in[0,T^*]$.
Further we let $\theta$ be a smooth truncation of the identity.
Precisely, $\theta(-s)=-\theta(s)$ for all $s$, $\theta(s)=s$ for
$0\leq s\leq\delta/2$ and $\theta(s)=\delta$ for $s>\delta$.
Now we let
\begin{equation}
  \label{eq:def-vartheta}
  \vartheta\coloneqq\theta\circ\sdist
\end{equation}
and define the relative energy
\begin{equation}
  \label{eq:def-entropy-functional}
  \mathcal{E}(t)\coloneqq\int_{\Sigma(t)}(1-\nu(\cdot,t)\cdot\xi(\cdot,t))\,d\mathcal{H}^{d-1},
\end{equation}
and the volume error
\begin{equation}
  \label{eq:def-volume-error}
  \mathcal{F}(t)
  \coloneqq\int_{\mathbb{R}^2}(\chi_{\Omega(t)}-\chi_{\Omega^*(t)})\vartheta(\cdot,t)\,dx.
\end{equation}
By construction we have
\begin{align*}
  \mathcal{F}(t)
  &=\int_{\mathbb{R}^2}|\chi_{\Omega(t)}-\chi_{\Omega^*(t)}||\vartheta(\cdot,t)|\,dx.
\end{align*}

To prove weak-strong uniqueness, we will show an estimate of the form
\begin{align*}
  \mathcal{E}(T')+\mathcal{F}(T')\leq Ce^{CT'}(\mathcal{E}(0)+\mathcal{F}(0)).
\end{align*}
By Gronwall's inequality it suffices to show
\begin{equation}
  \label{eq:suffices-gronwall}
  \mathcal{E}(T')+\mathcal{F}(T')
  -\mathcal{E}(0)-\mathcal{F}(0)
  \leq C\int_0^{T'}\mathcal{E}(t)+\mathcal{F}(t)\,dt
  \qquad\text{for a.e.\ $T'\in[0,\min\{T,T^*\}]$,}
\end{equation}
which formally means
\begin{equation}
  \label{eq:gronwall-formal}
  \frac{d}{dt}(\mathcal{E}(t)+\mathcal{F}(t))\leq C(\mathcal{E}(t)+\mathcal{F}(t))
  \quad\text{for a.e.\ $t\in[0,\min\{T,T^*\}]$.}
\end{equation}

First we perform some formal computations in the spirit of the differential
inequality~\eqref{eq:gronwall-formal} which we will turn into a rigorous
integral inequality in form of~\eqref{eq:suffices-gronwall}.
We have
\begin{equation}
  \label{eq:d-dt-E}
  \frac{d}{dt}\mathcal{E}(t)
  =\frac{d}{dt}\mathcal{H}^{d-1}(\Sigma(t))
  -\frac{d}{dt}\int_{\Sigma(t)}\nu\cdot\xi\,d\mathcal{H}^{d-1}.
\end{equation}
For the first term, we use the formal differential inequality corresponding to~\eqref{eq:area-decay-assumption}:
\begin{align*}
  \frac{d}{dt}\mathcal{H}^{d-1}(\Sigma(t))
  &\leq-\frac{1}{2}\int_{\Sigma(t)}|\nabla_{\Sigma(t)}H|^2\,d\mathcal{H}^{d-1}
  -\frac{1}{2}\int_{\Sigma(t)}|\nabla_{\Sigma(t)}\varphi_V|^2\,d\mathcal{H}^{d-1}.
\end{align*}
We compute, using the formal differential identity corresponding to~\eqref{eq:def-V-weak},
\begin{align*}
  -\frac{d}{dt}\int_{\Sigma(t)}\nu\cdot\xi\,d\mathcal{H}^{d-1}
  &=-\frac{d}{dt}\int_{\Omega(t)}\nabla\cdot\xi\,dx \\
  &=-\int_{\Omega(t)}\nabla\cdot(\partial_t\xi)\,dx
    -\int_{\Sigma(t)}V\nabla\cdot\xi\,d\mathcal{H}^{d-1} \\
  &=-\int_{\Sigma(t)}\nu\cdot(\partial_t\xi)+(\nabla\cdot\xi)V\,d\mathcal{H}^{d-1}.
\end{align*}
Now we complete a square and combine the above
\begin{equation}
  \label{eq:first-estimate-dt-E}
  \begin{split}
    \frac{d}{dt}\mathcal{E}(t)
    \leq&\,\,-\frac{1}{2}\int|\nabla_{\Sigma(t)}H|^2\,d\mathcal{H}^{d-1}
          -\frac{1}{2}\int_{\Sigma(t)}|\nabla_{\Sigma(t)}\varphi_V|^2\,d\mathcal{H}^{d-1} \\
        &\,\,-\int_{\Sigma(t)}\nu\cdot(\partial_t\xi)\,d\mathcal{H}^{d-1}
          -\int_{\Sigma(t)}(\nabla_{\Sigma(t)}(\nabla\cdot\xi))\cdot\nabla_{\Sigma(t)}\varphi_V\,d\mathcal{H}^{d-1} \\
    =&\,\,-\frac{1}{2}\int_{\Sigma(t)}|\nabla_{\Sigma(t)}H|^2\,d\mathcal{H}^{d-1}
       -\frac{1}{2}\int_{\Sigma(t)}|\nabla_{\Sigma(t)}\varphi_V-\nabla_{\Sigma(t)}(\nabla\cdot\xi)|^2\,d\mathcal{H}^{d-1} \\
        &\,\,-\int_{\Sigma(t)}\nu\cdot(\partial_t\xi)\,d\mathcal{H}^{d-1}
          +\frac{1}{2}\int_{\Sigma(t)}|\nabla_{\Sigma(t)}(\nabla\cdot\xi)|^2\,d\mathcal{H}^{d-1}.
  \end{split}
\end{equation}

The main goal of the present work is to estimate the terms on the right hand
side by $C(\mathcal{E}(t)+\mathcal{F}(t))$.
To this end, we extend the differential operator $\nabla_{\Sigma^*(t)}$ to
$\mathbb{R}^d$ and give some estimates between the geometric quantities on
$\Sigma(t)$ and $\Sigma^*(t)$.
\begin{definition}\label{def:extension-quantities-Sigma-star}
  Let $\eta:\mathbb{R}\rightarrow[0,\infty)$ be a cutoff function such that
  $\eta=1$ on $[-\delta,\delta]\supseteq\supp\zeta$ and
  $\eta(s)=0$ for $|s|\geq 2\delta$.
  We extend $H^*,\,\nu^*$ and the differential operator $\nabla_{\Sigma^*(t)}$
  to all of $\mathbb{R}^d$ by
  \begin{align*}
    H^*\coloneqq\eta\Delta\sdist,\quad
    \nu^*\coloneqq\eta\nabla\sdist,
    \quad\nabla_{\Sigma^*(t)}\coloneqq(\Id-\nu^*\otimes\nu^*)\nabla.
  \end{align*}
\end{definition}
By this definition, we have smooth functions $\nu^*,H^*$ defined on
$\mathbb{R}^d\times[0,T^*]$ such that the restriction to $\Sigma^*$ is the
normal and curvature, respectively.
Further, the operator $\nabla_{\Sigma^*(t)}$ is defined on $\mathbb{R}^d$ and
coincides with the surface gradient on $\Sigma^*(t)$.
We suppress the dependence on $t$ in the following lemma.

\begin{lemma}\label{lem:various-estimates}
  For any weakly differentiable function $u$ on $\mathbb{R}^d$ it holds
  \begin{equation}
    \label{eq:estimate-operators}
    \zeta|\nabla_{\Sigma} u-\nabla_{\Sigma^*}u|^2
    \leq 2|\nabla u|^2\zeta|\nu-\nu^*|^2
    \quad\text{a.e.\ on $\Sigma$.}
  \end{equation}
  Furthermore
  \begin{equation}
    \label{eq:estimate-nu-minus-nu-star-squared}
    \zeta|\nu-\nu^*|^2
    \leq 2(1-\nu\cdot\xi)\quad\text{on $\Sigma$,}
  \end{equation}
  and
  \begin{equation}
    \label{eq:estimate-nabla-Sigma-sdist}
    |\nabla_{\Sigma}\sdist|^2
    \leq 2(1-\nu\cdot\xi)\quad\text{on $\Sigma$.}
  \end{equation}
\end{lemma}

\begin{proof}
  We compute
  \begin{align*}
    |(\nu\cdot\nabla u)\nu-(\nu^*\cdot\nabla u)\nu^*|^2
    &\leq|(\nu\cdot\nabla u)(\nu-\nu^*)|^2
      +|((\nu-\nu^*)\cdot\nabla u)|^2
      \leq 2|\nabla u|^2|\nu-\nu^*|^2.
  \end{align*}
  Furthermore
  \begin{align*}
    \zeta|\nu-\nu^*|^2
    &=2\zeta(1-\nu\cdot\nu^*)
      \leq 2(1-\nu\cdot\xi),
  \end{align*}
  and
  \begin{align*}
    |\nabla_\Sigma\sdist|^2
    &=|\nabla\sdist|^2-(\nu\cdot\nabla\sdist)^2
      \leq|\nabla\sdist|^2-(\zeta\circ\sdist)^2(\nu\cdot\nabla\sdist)^2
      \leq 2(1-\nu\cdot\xi).
  \end{align*}
  Finally
  \begin{align*}
    &\int_{\Sigma}(\zeta u-v)^2\,d\mathcal{H}^{d-1} \\
    &\leq
    (\mathcal{H}^{d-1}(\Sigma))^2\int_{\Sigma}(\nabla_{\Sigma}(\zeta u)-\nabla_{\Sigma}v)^2\,d\mathcal{H}^{d-1}
      +\int_{\Sigma} c^2\,d\mathcal{H}^{d-1} \\
    &\lesssim(\mathcal{H}^{d-1}(\Sigma))^2
      \left( \int_{\Sigma}(\zeta\nabla_{\Sigma}u-\nabla_{\Sigma}v)^2\,d\mathcal{H}^{d-1}
                  +\int_{\Sigma}(\zeta'u)^2(\nabla_{\Sigma}\sdist)^2\,d\mathcal{H}^{d-1}
      \right) +\int_{\Sigma} c^2\,d\mathcal{H}^{d-1} \\
    &\leq(\mathcal{H}^{d-1}(\Sigma))^2\int_{\Sigma}(\zeta\nabla_{\Sigma}u-\nabla_{\Sigma}v)^2\mathcal{H}^{d-1}
      +\int_{\Sigma} c^2\,d\mathcal{H}^{d-1} +(\mathcal{H}^{d-1}(\Sigma))^2\|\zeta'u\|_\infty^2\mathcal{E}. \qedhere
  \end{align*}
\end{proof}

\section{Stability for stationary solutions}
\label{sec:stability-for-stationary-sol}

We are now in the position to prove our first main result, which concerns the
stability of stationary solutions.

\begin{theorem}\label{thm:stability-stationary-sol}
  Let $\Sigma^*$ be a smooth stationary solution to surface diffusion in the sense of
  Definition~\ref{def:strong}.
  Then
  \begin{equation}
    \label{eq:thm-stationary-gronwall-estimate}
    \mathcal{E}(T')+\mathcal{F}(T')\leq Ce^{CT'}(\mathcal{E}(0)+\mathcal{F}(0))
    \quad\text{for a.e.\ $T'\in[0,\min\{T,T^*\}]$}
  \end{equation}
  for any weak solution $\Sigma$ to surface diffusion according to
  Definition~\ref{def:weak}.
  In particular, if $\Omega(0)=\Omega^*(0)$ up to Lebesgue null sets, then
  $\Omega(t)=\Omega^*(t)$ up to Lebesgue null sets for a.e.\ $t\in[0,\min\{T,T^*\}]$.
\end{theorem}

Before we prove the theorem, we prove the following lemma.

\begin{lemma}\label{lem:estimate-grad-sigma-div-xi}
  Let $\Sigma^*$ be a stationary solution to surface diffusion according to
  Definition~\ref{def:strong} and let $\Sigma$ be a weak solution to surface
  diffusion according to Definition~\ref{def:weak}.
  Then there exists a constant $C<\infty$ depending only on $\Sigma^*$ such that
  \begin{equation}
    \label{eq:estimate-grad-sigma-div-xi}
    \int_{\Sigma(t)}|\nabla_{\Sigma(t)}(\nabla\cdot\xi)|^2\,d\mathcal{H}^{d-1}
    \leq C\mathcal{E}(t)\quad\text{for a.e.\ $t\in[0,\min\{T,T^*\}]$.}
  \end{equation}
\end{lemma}

\begin{proof}
  We compute
  \begin{align*}
    \nabla_{\Sigma(t)}(\nabla\cdot\xi)
    &=(\zeta\circ\sdist)\nabla_{\Sigma(t)}(\Delta\sdist)
      +(\zeta'\circ\sdist)(\Delta\sdist)\nabla_{\Sigma(t)}\sdist
      +(\zeta''\circ\sdist)\nabla_{\Sigma(t)}\sdist.
  \end{align*}
  We note that $\Delta\sdist=H^*$ on $\supp\zeta\circ\sdist$.
  By~\eqref{eq:estimate-operators} and~\eqref{eq:estimate-nabla-Sigma-sdist}
  we have $|\nabla_{\Sigma(t)}\sdist|^2\leq 2(1-\nu\cdot\xi)$ and
  $|\nabla_{\Sigma(t)}H^*-\nabla_{\Sigma^*(t)}H^*|^2\lesssim 1-\nu\cdot\xi$.
  Hence,
  \begin{align*}
    \int_{\Sigma(t)}|\nabla_{\Sigma(t)}(\nabla\cdot\xi)|^2\,d\mathcal{H}^{d-1}
    &\lesssim\mathcal{E}(t)+\int_{\Sigma(t)}(\zeta\circ\sdist)^2|\nabla_{\Sigma^*(t)}H^*|^2\,d\mathcal{H}^{d-1}.
  \end{align*}
  Next observe that, since $\zeta\circ\sdist=1$ on $\Sigma^*$
  and $\Delta_{\Sigma^*}H^*=V^*=0$, we have
  \begin{align*}
    \int_{\Sigma^*(t)}|\nabla_{\Sigma^*(t)}H^*|^2\,d\mathcal{H}^{d-1}
    &=-\int_{\Sigma^*(t)}H^*\Delta_{\Sigma^*(t)}H^*\,d\mathcal{H}^{d-1}
      =0.
  \end{align*}
  Hence $|\nabla_{\Sigma^*(t)}H^*(\cdot,t)|=0$ on $\Sigma^*(t)$.
  Since $H^*$ is $C^2$, we have
  $|\nabla_{\Sigma^*(t)}H^*|=O(\sdist(\cdot,t))$ in $\mathcal{U}_\delta(t)$.
  Therefore
  \begin{align*}
    \int_{\Sigma(t)}(\zeta\circ\sdist)^2|\nabla_{\Sigma^*(t)}H^*(\cdot,t)|^2\,d\mathcal{H}^{d-1}
    &\lesssim\int_{\mathcal{U}_\delta(t)}\sdist(\cdot,t)^2\,dx
      \lesssim\mathcal{E}(t). \qedhere
  \end{align*}
\end{proof}

\begin{proof}[Proof of Theorem~\ref{thm:stability-stationary-sol}]
  We start by estimating $\mathcal{E}(T')-\mathcal{E}(0)$ in the spirit of the
  formal computation that lead to~\eqref{eq:first-estimate-dt-E}.
  We write
  \begin{align*}
    \mathcal{E}(T')-\mathcal{E}(0)
    =\,\,&\mathcal{H}^{d-1}(\Sigma(T'))-\mathcal{H}^{d-1}(\Sigma(0)) \\
    &-\left(
      \int_{\Sigma(T')}\nu(\cdot,T')\cdot\xi(\cdot,T')\,d\mathcal{H}^{d-1}
      -\int_{\Sigma(0)}\nu(\cdot,0)\cdot\xi(\cdot,0)\,d\mathcal{H}^{d-1}
      \right).
  \end{align*}
  Using~\eqref{eq:weak-formulation-kappa-curvature}
  and~\eqref{eq:phi-V-weak-formulation}, we can mirror the formal computations
  and obtain the following integral version
  \begin{equation}
    \label{eq:integral-ineq}
    \begin{split}
      &\mathcal{E}(T')-\mathcal{E}(0) \\
      \leq\,\,
      &-\frac{1}{2}\int_0^{T'}\int_{\Sigma(t)}|\nabla_{\Sigma(t)}H|^2\,d\mathcal{H}^{d-1}\,dt \\
      &-\frac{1}{2}\int_0^{T'}\int_{\Sigma(t)}|\nabla_{\Sigma(t)}\varphi_V-\nabla_{\Sigma(t)}(\nabla\cdot\xi)|^2\,d\mathcal{H}^{d-1}\,dt \\
      &-\int_0^{T'}\int_{\Sigma(t)}\nu\cdot(\partial_t\xi)\,d\mathcal{H}^{d-1}\,dt
        +\frac{1}{2}\int_0^{T'}\int_{\Sigma(t)}|\nabla_{\Sigma(t)}(\nabla\cdot\xi)|^2\,d\mathcal{H}^{d-1}\,dt.
    \end{split}
  \end{equation}
  Using Lemma~\ref{lem:estimate-grad-sigma-div-xi} and the fact that
  $\partial_t\xi=0$, we have
  \begin{align*}
    \mathcal{E}(t)
    \leq&\,\,-\frac{1}{2}\int_0^{T'}\int_{\Sigma(t)}|\nabla_{\Sigma(t)}H|^2\,d\mathcal{H}^{d-1}\,dt \\
    &\,\,-\frac{1}{2}\int_0^{T'}\int_{\Sigma(t)}|\nabla_{\Sigma(t)}\varphi_V-\nabla_{\Sigma(t)}(\nabla\cdot\xi)|^2\,d\mathcal{H}^{d-1}\,dt
          +C\int_0^{T'}\mathcal{E}(t)\,dt.
  \end{align*}
  Now we estimate the bulk error $\mathcal{F}(t)$.
  We have
  \begin{align*}
    \mathcal{F}(T')-\mathcal{F}(0)
    &=\int_0^{T'}\int_{\Sigma(t)}\vartheta V\,d\mathcal{H}^{d-1}\,dt
      +\int_0^{T'}\int_{\mathbb{R}^d}(\chi-\chi^*)\partial_t\vartheta\,dx\,dt.
  \end{align*}
  The second term is zero because $\theta$ is independent of $t$.
  For the first term we use~\eqref{eq:phi-V-weak-formulation} and Young's
  inequality to obtain, for any $\eps>0$,
  \begin{align*}
    &\int_0^{T'}\int_{\Sigma(t)}\vartheta V\,d\mathcal{H}^{d-1}\,dt \\
    =\,\,&-\int_0^{T'}\int_{\Sigma(t)}(\nabla_{\Sigma(t)}\vartheta)\cdot\nabla_{\Sigma(t)}\varphi_V\,d\mathcal{H}^{d-1}\,dt \\
    \leq\,\,
    &\frac{1}{2\eps}\int_0^{T'}\int_{\Sigma(t)}|\nabla_{\Sigma(t)}\vartheta|^2\,d\mathcal{H}^{d-1}\,dt
     +\frac{\eps}{2}\int_0^{T'}\int_{\Sigma(t)}|\nabla_{\Sigma(t)}\varphi_V|^2\,d\mathcal{H}^{d-1}\,dt \\
    \leq\,\,&\frac{1}{2\eps}\int_0^{T'}\int_{\Sigma(t)}|\nabla_{\Sigma(t)}\vartheta|^2\,d\mathcal{H}^{d-1}\,dt \\
    &+\eps\int_0^{T'}\int_{\Sigma(t)}|\nabla_{\Sigma(t)}\varphi_V-\nabla_{\Sigma(t)}(\nabla\cdot\xi)|^2\,d\mathcal{H}^{d-1}\,dt \\
      &+\eps \int_0^{T'}\int_{\Sigma(t)}|\nabla_{\Sigma(t)}(\nabla\cdot\xi)|^2\,d\mathcal{H}^{d-1}\,dt.
  \end{align*}
  The first term is bounded by $\mathcal{E}(t)$, since $\vartheta$ is a
  Lipschitz function of the signed distance $\sdist$,
  cf.~\eqref{eq:def-vartheta},
  and $\int_{\Sigma(t)}|\nabla_{\Sigma(t)}\sdist|^2\lesssim\mathcal{E}$.
  The second term is absorbed by the second dissipation
  term in~\eqref{eq:first-estimate-dt-E}, for $\eps<\frac{1}{2}$.
  The third term is estimated by Lemma~\ref{lem:estimate-grad-sigma-div-xi}.
\end{proof}

\section{Stability of surface diffusion in 2D}

In this section we study the stability of non-stationary solutions to surface
diffusion, but we restrict ourselves to the 2D case.
Then the evolution reads
\begin{align*}
  V=\partial_s^2\kappa,
\end{align*}
where $\partial_s$ denotes differentiation with respect to the arc length
parameter $s$ on the surface $\Sigma(t)$ and $\kappa$ denotes the curvature of
$\Sigma(t)$.

We assume that $\Sigma(t)$ is a weak solution to surface diffusion flow
according to Definition~\ref{def:weak} and $d=2$.

Let $\tau,\,\tau^*$ denote
the tangent vectors and $\nu=J\tau,\,\nu^*=J\tau^*$ denote the outward unit
normal vectors on $\Sigma$ and $\Sigma^*$ respectively, where $J$ is the counter-clockwise rotation by 90 degrees.
Then the curvatures $\kappa$ and $\kappa^*$ are given by
$\kappa=-(\partial_s\tau)\cdot\nu$ and
$\kappa^*=-(\partial_{s^*}\tau^*)\cdot\nu^*$.
Note that the weak formulation of this equality is exactly~\eqref{eq:def-mean-curvature-weak}.

The second main result of the present work is the following stability and uniqueness statement for solutions to surface diffusion.

\begin{theorem}\label{thm:main-thm}
  Let $d=2$ and let $(\Sigma^*(t))_{t\in[0,T^*]}$ and $(\Sigma(t))_{t\in[0,T]}$ be
  a strong and a weak solution to surface diffusion according to
  Definition~\ref{def:strong}, respectively Definition~\ref{def:weak}.
  Then there exists a constant $C<\infty$ depending only on $\Sigma^*$ and
  $\mathcal{H}^1(\Sigma(0))$ such that
  \begin{equation}
    \label{eq:main-thm-gronwall-estimate}
    \mathcal{E}(T')+\mathcal{F}(T')\leq Ce^{CT'}(\mathcal{E}(0)+\mathcal{F}(0))
    \quad\text{for a.e.\ $T'\in[0,\min\{T,T^*\}]$.}
  \end{equation}
  In particular, if $\Omega(0)=\Omega^*(0)$ up to Lebesgue null sets, then
  $\Omega(t)=\Omega^*(t)$ up to Lebesgue null sets for a.e.\ $t\in[0,\min\{T,T^*\}]$.
\end{theorem}

As before, it suffices to show that
\begin{equation}
  \label{eq:estimate-entropy}
  \mathcal{E}(T')+\mathcal{F}(T')-\mathcal{E}(0)-\mathcal{F}(0)
  \leq C\int_0^{T'}(\mathcal{E}(t)+\mathcal{F}(t))\,dt
  \quad\text{for a.e.\ $T'\in[0,\min\{T,T^*\}]$.}
\end{equation}

We again extend the geometric quantities
$\kappa^*,\,\nu^*,\,\tau^*$ and the operator $\partial_{s^*}$ to all of
$\mathbb{R}^2$ as in Definition~\ref{def:extension-quantities-Sigma-star}.

\begin{definition}
  Let $\eta:\mathbb{R}\rightarrow[0,\infty)$ be as in
  Definition~\ref{def:extension-quantities-Sigma-star},
  and define $\kappa^*,\,\tau^*,\,\nu^*$ and $\partial_{s^*}$ by
  \begin{align*}
    \kappa^*\coloneqq\eta\Delta\sdist,\quad
    \nu^*\coloneqq\eta\nabla\sdist,
    \quad\tau^*\coloneqq J^{-1}\nu^*,
    \quad\partial_{s^*}\coloneqq\tau^*\cdot\nabla.
  \end{align*}
\end{definition}

Then, on $\mathcal{U}_\delta$,
\begin{equation}
  \label{eq:gradient-pi-star}
  \begin{split}
    \nabla\pi^*
    &=\Id-\nu^*\otimes\nu^*-\sdist\nabla\nu^*
    =\tau^*\otimes\tau^*-\sdist\nabla\nu^*
  \end{split}
\end{equation}
and therefore
\begin{equation}
  \label{eq:partial-star-pi-star}
  \partial_{s^*}\pi^*
  =(1+\sdist\kappa^*)\tau^*\quad\text{on $\mathcal{U}_\delta$.}
\end{equation}

Next, we recall the 2D version of Lemma~\ref{lem:various-estimates}.
We suppress the $t$-dependence in the following statement and the proof of
Theorem~\ref{thm:main-thm}.

\begin{lemma}[{Lemma~\ref{lem:various-estimates} in 2D}]\label{lem:partial-s-minus-partial-s-star-bdd-by-E}
  For any weakly differentiable function $u$ on $\mathbb{R}^2$, we have
  \begin{equation}
    \label{eq:partial-s-minus-partial-s-star-bdd-tau-minus-tau-star}
    \zeta|\partial_s u-\partial_{s^*}u|^2
    \leq|\nabla u|^2\zeta|\tau-\tau^*|^2\quad\text{a.e.\ on $\Sigma$.}
  \end{equation}
  Furthermore
  \begin{equation}
    \label{eq:tau-tau-star-bdd-by-E}
    \zeta|\tau-\tau^*|^2
    =\zeta|\nu-\nu^*|^2
    \leq 2(1-\nu\cdot\xi)\quad\text{on $\Sigma$,}
  \end{equation}
  and
  \begin{equation}
    \label{eq:tau-cdot-nabla-sdist-bdd-by-E}
    |\tau\cdot\nabla\sdist|^2
    \leq 2(1-\nu\cdot\xi)\quad\text{on $\Sigma$.}
  \end{equation}
\end{lemma}

\begin{remark}
  From~\eqref{eq:partial-s-minus-partial-s-star-bdd-tau-minus-tau-star}
  and~\eqref{eq:tau-tau-star-bdd-by-E}, it follows that for any $v\in
  L^2(\Sigma),\,u\in\dot H^1(\mathbb{R}^2)\cap\dot W^{1,\infty}(\mathbb{R}^2)$,
  we have
  \begin{equation}
    \label{eq:replacing-partial-s-partial-s-star}
    \begin{split}
      \int_\Sigma\zeta(v-\partial_s u)^2\,ds
      &\leq 2\int_\Sigma\zeta(v-\partial_{s^*}u)^2\,ds
        +4\|\nabla u\|_\infty^2\mathcal{E}.
    \end{split}
  \end{equation}
  We refer to the above inequalities as
  \emph{replacing $\partial_s$ by $\partial_{s^*}$}
  or vice versa when performing estimates of this form.
  Furthermore
  \begin{equation}
    \label{eq:partial-s-sdist-bdd-by-E}
    \int_\Sigma(\partial_s(\zeta\circ\sdist))^2\,ds
    \leq\|\zeta'\|_\infty^2\int_\Sigma(\tau\cdot\nabla\sdist)^2\,ds
    \leq 2\|\zeta'\|_\infty^2\mathcal{E},
  \end{equation}
  and
  \begin{equation}
    \label{eq:partial-s-star-zeta-equals-zero}
    \partial_{s^*}(\zeta\circ\sdist)
    =(\zeta'\circ\sdist)\tau^*\cdot\nu^*
    =0.
  \end{equation}
  Finally, for any $u,v\in L^2(\Sigma)\cap L^\infty(\Sigma)$ we have
  \begin{equation}
    \label{eq:smuggle-in-zeta}
    \begin{split}
      \int_\Sigma(u-v)^2\,ds
      &\leq 2\int_\Sigma\zeta^2(u-v)^2\,ds
        +2\int_\Sigma(1-\zeta)^2(u-v)^2\,ds \\
      &\leq 2\int_\Sigma\zeta^2(u-v)^2\,ds+4\|u-v\|_\infty^2\mathcal{E}.
    \end{split}
  \end{equation}
\end{remark}

For each Jordan boundary $\Sigma_i(t)$ according to Definition~\ref{def:jordan-boundaries},
we let $\nu_i(\cdot,t)\coloneqq\nu(\cdot,t)|_{\Sigma_i(t)}$
be the restriction of the normal to $\Sigma_i(t)$.
Let $B:\mathbb{R}^2\times[0,T^*]\rightarrow\mathbb{R}^2$ be a smooth vector field.
At each time $t$ and each path component $\Sigma_i(t)$,
we let $\varphi_{\nu_i\cdot B}$ the zero-average potential $\varphi_{\nu_i\cdot B}(\cdot,t)$ such that
\begin{align*}
  \partial_{s}^2\varphi_{\nu_i\cdot B}(\cdot,t)
  &=\nu_i(\cdot,t)\cdot B(\cdot,t)-\fint_{\Sigma_i(t)}\nu_i(\cdot,t)\cdot B(\cdot,t)\,ds
  \quad\text{a.e.}
\end{align*}
with $\int_{\Sigma_i(t)}\varphi_{\nu_i\cdot B}(\cdot,t)\,ds=0$.
This solution exists by Lemma~\ref{lem:existence-sol-poisson-eq}.

We then define $\varphi_{\nu\cdot B}(\cdot,t):\Sigma(t)\rightarrow\mathbb{R}$ by
$\varphi_{\nu\cdot B}(\cdot,t)=\varphi_{\nu_i\cdot B}(\cdot,t)$ on $\Sigma_i(t)$.

Now we are ready to prove our second main result.

\begin{proof}[Proof of Theorem~\ref{thm:main-thm}]
  We continue from~\eqref{eq:first-estimate-dt-E} and keep the informal style.
  The differential inequalities can be turned into rigorous integral
  inequalities just as in the proof of Theorem~\ref{thm:main-thm}.
  In addition, we suppress the dependence on $t$.
  
  Adding zero to~\eqref{eq:first-estimate-dt-E} in terms
  of~\eqref{eq:weak-formulation-kappa-curvature} and using
  \begin{align*}
    \int_\Sigma\kappa\nu\cdot B\,ds
    &=\sum_{i=1}^{k(t)}\int_{\Sigma_i}\kappa \left(
      \nu_i\cdot B-\fint_{\Sigma_i}\nu_i\cdot B\,ds'
      \right)\,ds
      +\sum_{i=1}^{k(t)}\left( \int_{\Sigma_i}\kappa\,ds \right)\fint_{\Sigma_i}\nu_i\cdot B\,ds \\
    &=-\int_\Sigma\partial_s\kappa\partial_s\varphi_{\nu\cdot B}\,ds
      +\sum_{i=1}^{k(t)}\left( \int_{\Sigma_i}\kappa\,ds \right)\fint_{\Sigma_i}\nu_i\cdot B\,ds,
  \end{align*}
  we complete another square
  \begin{align*}
    \frac{d}{dt}\mathcal{E}(t)
    \leq&\,\,-\frac{1}{2}\int_\Sigma(\partial_s\kappa)^2\,ds
       -\frac{1}{2}\int_\Sigma(\partial_s\varphi_V-\partial_s(\nabla\cdot\xi))^2\,ds \\
     &\,\,-\int_\Sigma\nu\cdot(\partial_t\xi)\,ds
       +\frac{1}{2}\int_\Sigma(\partial_s(\nabla\cdot\xi))^2\,ds \\
     &\,\,+\int_\Sigma\nabla\cdot B\,ds
       -\int_\Sigma\nu\cdot\nabla B\nu\,ds \\
     &\,\,+\int_\Sigma\partial_s\kappa\partial_s\varphi_{\nu\cdot B}\,ds
       -\sum_{i=1}^{k(t)}\left( \int_{\Sigma_i}\kappa\,ds \right)\fint_{\Sigma_i}\nu_i\cdot B\,ds \\
    =&\,\,-\frac{1}{2}\int_\Sigma(\partial_s\kappa-\partial_s\varphi_{\nu\cdot B})^2\,ds
       -\frac{1}{2}\int_\Sigma(\partial_s\varphi_V-\partial_s(\nabla\cdot\xi))^2\,ds \\
     &\,\,-\int_\Sigma\nu\cdot(\partial_t\xi)\,ds
       +\frac{1}{2}\int_\Sigma(\partial_s(\nabla\cdot\xi))^2\,ds \\
     &\,\,+\int_\Sigma\nabla\cdot B\,ds
       -\int_\Sigma\nu\cdot\nabla B\nu\,ds
       -\frac{1}{2}\int_\Sigma(\partial_s\varphi_{\nu\cdot B})^2\,ds \\
     &\,\,-\sum_{i=1}^{k(t)}\left( \int_{\Sigma_i}\kappa\,ds \right)\fint_{\Sigma_i}\nu_i\cdot B\,ds.
  \end{align*}
  Now we complete one more square by adding zero in terms of
  $\int_\Sigma(\partial_s(\nabla\cdot\xi))\partial_s\varphi_{\nu\cdot B}\,ds$ and
  use the definition of $\varphi_{\nu\cdot B}$ to obtain
  \begin{align*}
    \frac{d}{dt}\mathcal{E}(t)
    \leq&\,\,-\frac{1}{2}\int_\Sigma(\partial_s\kappa-\partial_s\varphi_{\nu\cdot B})^2\,ds
       -\frac{1}{2}\int_\Sigma(\partial_s\varphi_V-\partial_s(\nabla\cdot\xi))^2\,ds \\
     &\,\,+\frac{1}{2}\int_\Sigma(\partial_s(\nabla\cdot\xi))^2\,ds
       +\frac{1}{2}\int_\Sigma(\partial_s\varphi_{\nu\cdot B})^2\,ds \\
     &\,\,+\int_\Sigma(\partial_s(\nabla\cdot\xi))\partial_s\varphi_{\nu\cdot B}\,ds
       -\int_\Sigma(\partial_s(\nabla\cdot\xi))\partial_s\varphi_{\nu\cdot B}\,ds \\
     &\,\,-\int_\Sigma\nu\cdot(\partial_t\xi)\,ds
       +\int_\Sigma\nabla\cdot B\,ds
       -\int_\Sigma\nu\cdot\nabla B\nu\,ds \\
     &\,\,-\sum_{i=1}^{k(t)}\left( \int_{\Sigma_i}\kappa\,ds \right)\fint_{\Sigma_i}\nu_i\cdot B\,ds \\
    =&\,\,-\frac{1}{2}\int_\Sigma(\partial_s\kappa-\partial_s\varphi_{\nu\cdot B})^2\,ds
       -\frac{1}{2}\int_\Sigma(\partial_s\varphi_V-\partial_s(\nabla\cdot\xi))^2\,ds \\
     &\,\,+\frac{1}{2}\int_\Sigma(\partial_s\varphi_{\nu\cdot B}-\partial_s(\nabla\cdot\xi))^2\,ds \\
     &\,\,-\int_\Sigma\nu\cdot B(\nabla\cdot\xi)\,ds
       +\sum_{i=1}^{k(t)}\left( \int_{\Sigma_i}(\nabla\cdot\xi)\,ds \right)\fint_{\Sigma_i}\nu_i\cdot B\,ds \\
     &\,\,-\int_\Sigma\nu\cdot(\partial_t\xi)\,ds
       +\int_\Sigma\nabla\cdot B\,ds
       -\int_\Sigma\nu\cdot\nabla B\nu\,ds \\
     &\,\,-\sum_{i=1}^{k(t)}\left( \int_{\Sigma_i}\kappa\,ds \right)\fint_{\Sigma_i}\nu_i\cdot B\,ds.
  \end{align*}
  We recall that every exact differential form is closed and apply Gauss' theorem to
  $B\wedge\xi\coloneqq B\otimes\xi-\xi\otimes B$:
  \begin{equation}
    \label{eq:computation-Gauss-B-wedge-xi}
    \begin{split}
      0&=\int_{\Omega}\nabla\cdot(\nabla\cdot(B\wedge\xi))\,dx \\
       &=\int_{\Sigma}\nu\cdot(\nabla\cdot(B\wedge\xi))\,ds \\
       &=\int_\Sigma\bigg[(\nabla\cdot\xi)\nu\cdot B
         +\nu\cdot(\xi\cdot\nabla)B
         -(\nabla\cdot B)\nu\cdot\xi
         -\nu\cdot(B\cdot\nabla)\xi
         \bigg]\,ds.
    \end{split}
  \end{equation}
  Now we add zero in terms of~\eqref{eq:computation-Gauss-B-wedge-xi} and
  rearrange terms
  \begin{align*}
    \frac{d}{dt}\mathcal{E}(t)
    \leq&\,\,-\frac{1}{2}\int_\Sigma(\partial_s\kappa-\partial_s\varphi_{\nu\cdot B})^2\,ds
       -\frac{1}{2}\int_\Sigma(\partial_s\varphi_V-\partial_s(\nabla\cdot\xi))^2\,ds \\
     &\,\,+\frac{1}{2}\int_\Sigma(\partial_s\varphi_{\nu\cdot B}-\partial_s(\nabla\cdot\xi))^2\,ds
       -\int_\Sigma\nu\cdot B(\nabla\cdot\xi)\,ds \\
     &\,\,-\int_\Sigma\nu\cdot(\partial_t\xi)\,ds
       +\int_\Sigma\nabla\cdot B\,ds
       -\int_{\Sigma}\nu\cdot\nabla B\nu\,ds \\
     &\,\,+\sum_{i=1}^{k(t)}\left( \int_{\Sigma_i}(\nabla\cdot\xi)\,ds-\int_{\Sigma_i}\kappa\,ds \right)
       \fint_{\Sigma_i}\nu_i\cdot B\,ds \\
     &\,\,+\int_\Sigma\bigg[
       (\nabla\cdot\xi)\nu\cdot B
       +\nu\cdot(\xi\cdot\nabla)B
       -(\nabla\cdot B)\nu\cdot\xi
       -\nu\cdot(B\cdot\nabla)\xi
       \bigg]\,ds \\
    =&\,\,-\frac{1}{2}\int_\Sigma(\partial_s\kappa-\partial_s\varphi_{\nu\cdot B})^2\,ds
       -\frac{1}{2}\int_\Sigma(\partial_s\varphi_V-\partial_s(\nabla\cdot\xi))^2\,ds \\
     &\,\,+\frac{1}{2}\int_\Sigma(\partial_s\varphi_{\nu\cdot B}-\partial_s(\nabla\cdot\xi))^2\,ds
       +\int_\Sigma(\nabla\cdot B)(1-\nu\cdot\xi)\,ds \\
     &\,\,-\int_\Sigma(\nu-\xi)\cdot\nabla B(\nu-\xi)\,ds
       -\frac{1}{2}\int_\Sigma(\partial_t|\xi|^2+(B\cdot\nabla)|\xi|^2)\,ds \\
     &\,\,-\int_\Sigma(\partial_t\xi+(B\cdot\nabla)\xi+(\nabla B)^T\xi)\cdot(\nu-\xi)\,ds \\
     &\,\,+\sum_{i=1}^{k(t)}\left( \int_{\Sigma_i}(\nabla\cdot\xi)\,ds-\int_{\Sigma_i}\kappa\,ds \right)
       \fint_{\Sigma_i}\nu_i\cdot B\,ds \\
  \end{align*}
  All terms are either non-positive or directly estimated by $C(1-\nu\cdot\xi)$,
  except for
  \begin{equation}
    \label{eq:terms-to-estimate-1}
    \begin{split}
      &\,\,\frac{1}{2}\int_\Sigma(\partial_s\varphi_{\nu\cdot B}-\partial_s(\nabla\cdot\xi))^2\,ds
        -\frac{1}{2}\int_\Sigma(\partial_t|\xi|^2+(B\cdot\nabla)|\xi|^2)\,ds \\
      &\,\,-\int_\Sigma(\partial_t\xi+(B\cdot\nabla)\xi+(\nabla B)^T\xi)\cdot(\nu-\xi)\,ds \\
      &\,\,+\sum_{i=1}^{k(t)}\left( \int_{\Sigma_i}(\nabla\cdot\xi)\,ds-\int_{\Sigma_i}\kappa\,ds \right)
        \fint_{\Sigma_i}\nu_i\cdot B\,ds.
    \end{split}
  \end{equation}
  To estimate those terms, we choose a suitable vector field $B$ according to
  Corollary~\ref{cor:existence-B} below.
  
  \begin{remark}
    Until this point, we have not used the fact that we are in $\mathbb{R}^2$.
    Everything would have worked the same way in $\mathbb{R}^d$.
    However, we will soon use the Gauss--Bonnet theorem, which only applies to the
    Gauss curvature, not the mean curvature.
    Further we will have terms of the form $\partial_s\partial_{s^*}\varphi_f$ for
    some smooth potential $\varphi_f$ with $\partial_{s^*}^2\varphi_f=f$.
    In higher dimensions, we would want to estimate terms like
    $\nabla_{\Sigma}\nabla_{\Sigma^*}\varphi_f$ instead.
    Since $\nabla_{\Sigma^*} \varphi_f$ is not a tangential vector field on
    $\Sigma$, this creates additional curvature terms.
  \end{remark}

  We estimate the last two terms of~\eqref{eq:terms-to-estimate-1}.
  By the Gauss--Bonnet theorem, we have
  \begin{align*}
    \sum_{i=1}^{k(t)}\left( \int_{\Sigma_i(t)}\kappa\,ds  \right)\fint_{\Sigma_i(t)}\nu_i\cdot B\,ds
    =2\pi\sum_{i=1}^{k(t)}\frac{1}{\mathcal{H}^1(\Sigma_i(t))}\int_{\Sigma_i(t)}\nu_i\cdot B\,ds.
  \end{align*}
  Hence the last term in~\eqref{eq:terms-to-estimate-1} is estimated by
  Lemma~\ref{lem:estimate-nu-dot-B}.
  To estimate the second term in~\eqref{eq:terms-to-estimate-1},
  we introduce the trivial extension of the normal velocity vector $V^*\nu^*$ on $\mathcal{U}_\delta$ defined by
  \begin{align*}
    \overline B(\cdot,t)=(\eta\circ\sdist)(B\circ\pi^*).
  \end{align*}
  Then
  \begin{equation}
    \label{eq:overline-B-eq-V-star}
    \overline B\cdot\nabla\sdist
    =(\partial_{s^*}^2\Delta\sdist\circ\pi^*)
    =V^*\circ\pi^*
    \quad\text{on $\mathcal{U}_\delta$.}
  \end{equation}
  We compute
  \begin{equation}
    \label{eq:estimate-partial-t-plus-B-dot-nabla-|xi|^2}
    (\partial_t+(B\cdot\nabla))|\xi|^2
    =(\partial_t+(\overline B\cdot\nabla))(\zeta^2\circ\sdist)
    +(B-\overline B)\cdot\nabla(\zeta^2\circ\sdist).
  \end{equation}
  The first term is exactly zero on $\mathcal{U}_\delta$, since $\zeta^2$ is a function of
  the signed distance, i.e.
  \begin{align*}
    \partial_t(\zeta^2\circ\sdist)+\overline B\cdot\nabla(\zeta^2\circ\sdist)
    =0,
  \end{align*}
  cf.~\cite[Lem.\ 1.19]{Laux2021}.
  For the second term we use the Lipschitz estimate
  \begin{equation}
    \label{eq:B-overline-B-O-sdist}
    |B-\overline B|
    =|B\circ(\id-\pi^*)|
    \leq C|\sdist|
    \quad\text{on $\mathcal{U}_\delta$}
  \end{equation}
  and compute
  \begin{align*}
    |\nabla(\zeta^2\circ\sdist)|\leq
    2(\zeta\circ\sdist)|\zeta'\circ\sdist||\nabla\sdist|
    \leq C|\sdist|,
  \end{align*}
  where we used that $\zeta'(0)=0$ and hence
  $\zeta'(\sdist)\lesssim|\sdist|$.
  Thus the last term on the RHS
  of~\eqref{eq:estimate-partial-t-plus-B-dot-nabla-|xi|^2} is $O(|\sdist|^2)$.

  Now we estimate the transport operator $\partial_t\xi+(B\cdot\nabla)\xi+(\nabla B)^T\xi$.
  We compute
  \begin{align*}
    \partial_t\xi+(B\cdot\nabla)\xi+(\nabla B)^T\xi
    =&\,\,\left( (\partial_t+(B\cdot\nabla))(\zeta\circ\sdist) \right)\nabla\sdist \\
     &\,\,+(\zeta\circ\sdist)(\partial_t\nabla\sdist+(B\cdot\nabla)\nabla\sdist+(\nabla B)^T\nabla\sdist).
  \end{align*}
  The first term on the RHS is again $O(\sdist)$, as above.
  The second term reads, after again smuggling in $\overline B$
  \begin{align*}
    \left( \partial_t\nabla\sdist+(B\cdot\nabla)\nabla\sdist+(\nabla B)^T\nabla\sdist \right)
    =&\,\,\left(
       \partial_t\nabla\sdist+(\overline B\cdot\nabla)\nabla\sdist+(\nabla\overline B)^T\nabla\sdist
       \right) \\
     &\,\,+(B-\overline B)\cdot\nabla^2\sdist
       +(\nabla B-\nabla\overline B)^T\nabla\sdist.
  \end{align*}
  The first term vanishes identically on $\supp\zeta\circ\sdist$.
  The second term is again $O(\sdist)$ by~\eqref{eq:B-overline-B-O-sdist} on $\supp\zeta\circ\sdist$.
  For the last term we compute
  \begin{align*}
    (\nabla B-\nabla\overline B)^T\nabla\sdist
    &=(\nabla B-(\nabla B\circ\pi^*)\nabla\pi^*)^T\nabla\sdist \\
    &=(\nabla B(\id-\pi^*))^T\nabla\sdist
      +(\Id-\nabla\pi^*)(\nabla B\circ\pi^*)^T\nabla\sdist.
  \end{align*}
  The first term on the RHS is again $O(\sdist)$ since $\nabla B$ is Lipschitz.
  Using $\nabla\pi^*=\Id-\nabla\sdist\otimes\nabla\sdist+O(\sdist)$,
  the second term reads
  \begin{align*}
    (\Id-\nabla\pi^*)(\nabla B\circ\pi^*)^T\nabla\sdist
    &=\nabla\sdist\otimes\nabla\sdist(\nabla B\circ\pi^*)^T\nabla\sdist
      +O(\sdist) \\
    &=(\nabla\sdist\cdot(\nabla B\circ\pi^*)^T\nabla\sdist)\nabla\sdist+O(\sdist).
  \end{align*}
  Since this error term is multiplied by $\zeta\circ\sdist$, we have
  \begin{align*}
    (\partial_t\xi+(B\cdot\nabla)\xi+(\nabla B)^T\xi)\cdot(\nu-\xi)
    &\lesssim O(\sdist^2)+\xi\cdot(\nu-\xi)
  \end{align*}
  and since $\xi\cdot(\nu-\xi)\leq\zeta(1-\zeta)$, also this term is estimated by
  $\mathcal{E}$.
  Note that we only used $\int_{\mathcal{U}_\delta}\sdist^2\,dx\leq C\mathcal{E}(t)$.
  This estimate does of course not hold in general when integrating over
  $\mathbb{R}^2$ instead of $\mathcal{U}_\delta$.

  Hence it remains to estimate
  \begin{equation}
    \label{eq:remaining-term}
    \int_\Sigma(\partial_s(\nabla\cdot\xi)-\partial_s\varphi_{\nu\cdot B})^2\,ds.
  \end{equation}

  To obtain an estimate for~\eqref{eq:remaining-term},
  we define another zero-average potential function
  $\varphi^*(\cdot,t):\Sigma^*(t)\rightarrow\mathbb{R}$ such that
  \begin{align*}
    \partial_{s^*}^2\varphi^*(\cdot,t)=\xi(\cdot,t)\cdot B(\cdot,t)=V^*(\cdot,t)
    \quad\text{on $\Sigma^*(t)$.}
  \end{align*}
  Note that $\fint_{\Sigma^*_i(t)}V^*(\cdot,t)\,ds^*=0$ for all $1\leq i\leq k^*$.
  We also define an extension
  $\varphi^*_{\xi\cdot B}(\cdot,t):\mathbb{R}^2\rightarrow\mathbb{R}$ by
  \begin{align*}
    \varphi^*_{\xi\cdot B}(x,t)
    \coloneqq\varphi^*(\pi^*(x,t),t)\zeta(\sdist(x,t)).
  \end{align*}
  Further, since $\xi\cdot B\in H^2(\Sigma^*)$ and $\Sigma^*$ is $C^3$, 
  we have $\varphi_{\xi\cdot B}\in C^2(\mathbb{R}^d)$.
  Using~\eqref{eq:partial-star-pi-star}, we have
  \begin{align*}
    \partial_{s^*}\varphi^*_{\xi\cdot B}
    &=(\zeta\circ\sdist)(\partial_{s^*}\varphi^*)\circ\pi^*(1+\sdist\kappa^*)
      +(\partial_{s^*}\zeta)(\varphi^*\circ\pi^*) \\
    &=(\zeta\circ\sdist)(1+\sdist\kappa^*)(\partial_{s^*}\varphi^*)\circ\pi^*
  \end{align*}
  because $\partial_{s^*}(\zeta\circ\sdist)=\zeta'\tau^*\cdot\nabla\sdist=0$ on
  $\mathcal{U}_\delta$, and hence
  \begin{align*}
    \partial_{s^*}^2\varphi^*_{\xi\cdot B}
    &=\xi\cdot B(1+\sdist\kappa^*)^2
      +(\zeta\circ\sdist)\sdist\partial_{s^*}\kappa^*((\partial_{s^*}\varphi^*_{\xi\cdot B})\circ\pi^*)
  \end{align*}
  on $\supp\zeta$.
  In particular we have for any $v\in L^2(\Sigma)$
  \begin{equation}
    \label{eq:laplace-star-phi-xi-B-bdd-by-xi-dot-B}
    \begin{split}
      \int_\Sigma(\partial_{s^*}^2\varphi^*_{\xi\cdot B}-v)^2\,ds
      &\lesssim\int_\Sigma(\xi\cdot B-v)^2\,ds
        +C\int_\Sigma(\zeta\circ\sdist)^2\sdist^2\,ds \\
      &\lesssim\int_\Sigma(\xi\cdot B-v)^2\,ds+\mathcal{E},
    \end{split}
  \end{equation}
  because $\sdist^2\leq 1-\zeta\leq 1-\nu\cdot\xi$ on $\supp\zeta$,
  where $C=C(\kappa^*,\varphi^*_{\xi\cdot B},\sdist)$.
  Now we smuggle
  $\zeta\partial_{s^*}\varphi^*_{\xi\cdot B}$
  into~\eqref{eq:remaining-term} and compute
  \begin{align*}
    &\,\,\int_\Sigma(\partial_s(\nabla\cdot\xi)-\partial_s\varphi_{\nu\cdot B})^2\,ds \\
    \lesssim&\,\,\int_\Sigma(\partial_s(\nabla\cdot\xi)-\zeta\partial_{s^*}\varphi^*_{\xi\cdot B})^2\,ds
              +\int_\Sigma(\zeta\partial_{s^*}\varphi^*_{\xi\cdot B}-\partial_s\varphi_{\nu\cdot B})^2\,ds \\
    \lesssim&\,\,\int_\Sigma(\partial_s(\nabla\cdot\xi)-\zeta\partial_{s^*}\varphi^*_{\xi\cdot B})^2\,ds
              +\int_\Sigma(\zeta\partial_{s}\varphi^*_{\xi\cdot B}-\partial_s\varphi_{\nu\cdot B})^2\,ds
              +C(\nabla\varphi^*_{\xi\cdot B})\mathcal{E}.
  \end{align*}
  We estimate the second term and write the integral as a sum over the connected components.
  On each term we subtract the average
  $c_i\coloneqq\fint_{\Sigma_i}\zeta\partial_{s^*}\varphi^*_{\xi\cdot B}$ and apply the Poincar\'{e}
  inequality~\eqref{eq:poincare-ineq} combined with~\eqref{eq:smuggle-in-zeta}
  and get that there exists a constant $C$ depending only on $\mathcal{H}^1(\Sigma(0))$ such that
  \begin{equation}\label{eq:comp-poincare}
     \begin{split}
      &\,\,\int_{\Sigma_i}(\zeta\partial_{s^*}\varphi^*_{\xi\cdot B}-\partial_s\varphi_{\nu_i\cdot B}-c_i)^2\,ds
      +\int_{\Sigma_i}\left( \fint_{\Sigma_i}\zeta\partial_{s^*}\varphi^*_{\xi\cdot B} \right)^2
    \\
    \lesssim
    &\,\,\int_{\Sigma_i}(\zeta\partial_{s^*}\varphi^*_{\xi\cdot B}-\partial_s\varphi_{\nu_i\cdot B}-c_i)^2\,ds
      +\int_{\Sigma_i}\left( \fint_{\Sigma_i}\zeta\partial_{s^*}\varphi^*_{\xi\cdot B}+\partial_s(\zeta\varphi^*_{\xi\cdot B}) \right)^2 \\
    \lesssim
    &\,\,C\int_{\Sigma_i}(\partial_s\partial_{s^*}\varphi^*_{\xi\cdot B}-\partial_s^2\varphi_{\nu_i\cdot B})^2\,ds
    +C(\varphi^*_{\xi\cdot B},\mathcal{H}^1(\Sigma(0)))\int_{\Sigma_i}1-\nu_i\cdot\xi\,ds \\
    \lesssim
    &\,\,C\int_{\Sigma_i}(\partial_{s^*}^2\varphi^*_{\xi\cdot B}-\partial_s^2\varphi_{\nu_i\cdot B})^2\,ds
      +C(\varphi^*_{\xi\cdot B},\mathcal{H}^1(\Sigma(0)))\int_{\Sigma_i}1-\nu_i\cdot\xi\,ds \\
    \lesssim
    &\,\,C\int_{\Sigma_i}\left( \xi\cdot B-\nu_i\cdot B \right)^2\,ds
      +C\int_{\Sigma_i}\left( \fint_{\Sigma_i}(\partial_{s^*}^2\Delta\sdist)\nu_i\cdot\xi\,ds' \right)^2\,ds \\
    &\,\,+C(\varphi^*_{\xi\cdot B},\mathcal{H}^1(\Sigma(0)))\int_{\Sigma_i}1-\nu_i\cdot\xi\,ds,
    \end{split}
      \end{equation}
  where we used~\eqref{eq:laplace-star-phi-xi-B-bdd-by-xi-dot-B} in the last
  step.
  Now we sum over the connected components $i=1,\dots,k(t)$.
  The contribution in the first and third term is bounded by
  \begin{align*}
    &\sum_{i=1}^{k(t)}\left(
      C(\mathcal{H}^1(\Sigma(0)))\int_{\Sigma_i(t)}((\xi-\nu_i)\cdot B)^2\,ds
      +C(\varphi^*_{\xi\cdot B},\mathcal{H}^1(\Sigma(0)))\int_{\Sigma_i}1-\nu_i\cdot\xi\,ds
      \right) \\
    \leq\,\,
    &C\sum_{i=1}^{k(t)}\left(
      \|B\|_\infty^2\int_{\Sigma_i(t)}|\nu-\xi|^2\,ds
      +C(\varphi^*_{\xi\cdot B})\int_{\Sigma_i}1-\nu_i\cdot\xi\,ds
      \right) \\
    \lesssim\,\,&C(\Sigma^*,\mathcal{H}^1(\Sigma(0)))\mathcal{E}.
  \end{align*}
  The second term in~\eqref{eq:comp-poincare} can be estimated by adding zero in
  terms of $\fint_{\Sigma_i}\partial_s(\zeta\partial_{s^*}\Delta\sdist)$ and using Jensen's inequality
  \begin{align*}
    &\,\,\int_{\Sigma_i}\left( \fint_{\Sigma_i}(\partial_{s^*}^2\Delta\sdist)\nu_i\cdot\xi\,ds' \right)^2\,ds \\
    \lesssim
    &\,\,\int_{\Sigma_i}\left( (\partial_{s^*}^2\Delta\sdist))\nu_i\cdot\xi-\partial_s(\zeta\partial_{s^*}\Delta\sdist) \right)^2\,ds \\
    \lesssim
    &\,\,\int_{\Sigma_i}\zeta^2\left( \partial_{s^*}^2\Delta\sdist-\partial_s\partial_{s^*}\Delta\sdist \right)^2\,ds \\
    &\,\,+\|\partial_{s^*}^2\kappa^*\|_\infty^2
      \int_{\Sigma_i}\zeta^2\left( 1-\nu_i\cdot\nu^* \right)^2\,ds
      +C(\kappa^*)\int_{\Sigma_i}|\nabla\sdist|^2\,ds.
  \end{align*}
  After summing over $i$, all terms above are again bounded by $\mathcal{E}$.
  Hence it remains to estimate
  \begin{align*}
    \int_\Sigma(\partial_{s}(\nabla\cdot\xi)-\zeta\partial_{s^*}\varphi^*_{\xi\cdot B})^2\,ds.
  \end{align*}
  We compute
  \begin{align*}
    \partial_s(\nabla\cdot\xi)
    &=\tau\cdot\nabla(\nabla\cdot\xi) \\
    &=\tau\cdot\left(
      (\zeta\circ\sdist)\nabla(\Delta\sdist)
      +(\zeta'\circ\sdist)(\Delta\sdist)\nabla\sdist
      +(\zeta''\circ\sdist)\nabla\sdist
      \right) \\
    &=(\zeta\circ\sdist)\partial_s\Delta\sdist+(\tau\cdot\nabla\sdist)
      \left(
      (\zeta'\circ\sdist)\Delta\sdist
      +(\zeta''\circ\sdist)
      \right).
  \end{align*}
  By Young's inequality we have
  \begin{align*}
    &\,\,\frac{1}{2}\int_\Sigma(\partial_s(\nabla\cdot\xi)-\zeta\partial_{s^*}\varphi^*_{\xi\cdot B})^2\,ds \\
    \leq&\,\,\int_\Sigma((\zeta\circ\sdist)\partial_s\Delta\sdist-\zeta\partial_{s^*}\varphi^*_{\xi\cdot B})^2\,ds
          +\int_\Sigma(\tau\cdot\nabla\sdist)^2\left(
          (\zeta'\circ\sdist)\Delta\sdist+(\zeta''\circ\sdist)
          \right)^2\,ds.
  \end{align*}
  We treat each summand separately.
  The last term can be bounded by $\mathcal{E}$
  by~\eqref{eq:tau-cdot-nabla-sdist-bdd-by-E}.
  Thus it remains to estimate the first term
  \begin{equation}
    \label{eq:term-to-estimate-partial-s-delta-sdist-minus-phi-xi-cdot-B}
    \int_\Sigma \left(
      \zeta\partial_s\Delta\sdist-\zeta\partial_{s^*}\varphi^*_{\xi\cdot B}
    \right)^2\,ds.
  \end{equation}
  Replacing $\partial_s$ with $\partial_{s^*}$, it is also sufficient to estimate
  \begin{equation}
    \label{eq:remaining-term-2}
    \int_\Sigma\zeta^2(\partial_{s^*}\Delta\sdist-\partial_{s^*}\varphi^*_{\xi\cdot B})^2\,ds.
  \end{equation}
  To this end, we again subtract the average and apply the Poincar\'{e} inequality to
  obtain
  \begin{align*}
    &\,\,\int_\Sigma(\zeta\partial_{s}\Delta\sdist-\zeta\partial_{s}\varphi^*_{\xi\cdot B})^2\,ds \\
    \lesssim
    &\,\,\mathcal{H}^1(\Sigma)^2\int_\Sigma\zeta^2(\partial_s\partial_{s^*}\Delta\sdist-\partial_s\partial_{s^*}\varphi^*_{\xi\cdot B})^2\,ds
      +\int_\Sigma\left( \fint_\Sigma\zeta\partial_{s^*}\Delta\sdist-\zeta\partial_{s^*}\varphi^*_{\xi\cdot B}\,ds' \right)^2\,ds \\
    &\,\,+C(\kappa^*,\varphi^*_{\xi\cdot B})\mathcal{E}.
  \end{align*}
  In the second term we add zero in terms of
  $\fint_\Sigma\partial_s(\zeta\Delta\sdist-\zeta\varphi^*_{\xi\cdot B})\,ds'$ and we again replace
  $\partial_s$ with $\partial_{s^*}$ and
  use~\eqref{eq:laplace-star-phi-xi-B-bdd-by-xi-dot-B} to obtain
  \begin{align*}
    &\,\,\int_\Sigma\zeta^2(\partial_s\partial_{s^*}\Delta\sdist-\partial_s\partial_{s^*}\varphi^*_{\xi\cdot B})^2\,ds
      +\int_\Sigma\left( \fint_\Sigma \zeta\partial_{s^*}\Delta\sdist-\zeta\partial_{s^*}\varphi^*_{\xi\cdot B}\,ds' \right)^2\,ds \\
    \lesssim
    &\,\,\int_\Sigma\zeta^2(\partial_{s^*}^2\Delta\sdist-\partial_{s^*}^2\varphi^*_{\xi\cdot B})^2\,ds
      +C(\kappa^*,\varphi^*_{\xi\cdot B})\mathcal{E} \\
    =
    &\,\,\int_\Sigma\zeta^2(\partial_{s^*}^2\Delta\sdist-\xi\cdot B)^2\,ds
      +C(\kappa^*,\varphi^*_{\xi\cdot B})\mathcal{E}.
  \end{align*}
  Now we smuggle in $\xi\cdot\overline B$ and obtain
  \begin{align*}
    \int_\Sigma\zeta^2(\partial_{s^*}^2\Delta\sdist-\xi\cdot B)^2\,ds
    &\lesssim
      \int_\Sigma\zeta^2(\partial_{s^*}^2\Delta\sdist-\xi\cdot\overline B)^2\,ds
      +\int_\Sigma\zeta^2(\xi\cdot(B-\overline B))^2\,ds \\
    &\lesssim
      \int_\Sigma\zeta^2(\partial_{s^*}^2\Delta\sdist-\xi\cdot\overline B)^2+O(\sdist^2)\,ds \\
    &\lesssim\int_\Sigma\zeta^2(\partial_{s^*}^2\Delta\sdist-(\partial_{s^*}^2\Delta\sdist)\circ\pi^*)^2+O(\sdist^2)\,ds
  \end{align*}
  where we used that $1-|\xi|^2\leq 2(1-\nu\cdot\xi)$ in the last line.
  The last term is bounded by $\mathcal{E}$ due to the Lipschitz estimate
  $|(\partial_{s^*}^2\Delta\sdist)\circ(\id-\pi^*)|\leq C|\sdist|$.
  This concludes the estimate for $\mathcal{E}$.

  Next, we treat the bulk error $\mathcal{F}(t)$.
  We recall that
  \begin{align*}
    |\partial_t\vartheta+(B\cdot\nabla)\vartheta|\leq C|\vartheta|.
  \end{align*}
  Using the above estimate and that $\vartheta(\cdot,t)=0$ on
  $\supp|\nabla\chi^*(\cdot,t)|=\Sigma^*(t)$,
  and $\int_{\Sigma(t)}(\nabla\cdot B)\vartheta\,ds\lesssim\mathcal{F}(t)$,
  we have
  \begin{equation}
    \label{eq:bulk-estimate-1}
    \begin{split}
      \frac{d}{dt}\mathcal{F}(t)
      =\,\,
      &\int_{\Sigma(t)}\vartheta V\,ds
        +\int_{\mathbb{R}^2}(\chi-\chi^*)\partial_t\vartheta\,dx \\
      \leq\,\,
      &\int_{\Sigma(t)}\vartheta V\,ds
        -\int_{\mathbb{R}^2}(\chi-\chi^*)(B\cdot\nabla)\vartheta\,dx
        +C\int_{\mathbb{R}^2}|\chi-\chi^*||\vartheta|\,dx \\
      =\,\,
      &C\mathcal{F}(t)
        +\int_{\Sigma(t)}\vartheta V\,ds
        -\int_{\Sigma(t)}\vartheta(\nu\cdot B)\,ds.
    \end{split}
  \end{equation}
  We compute
  \begin{align*}
    \int_{\Sigma(t)}\vartheta(V-\nu\cdot B)\,ds
    =\,\,
    &-\sum_{i=1}^{k(t)}
      \int_{\Sigma_i(t)}(\partial_s\vartheta)(\partial_s\varphi_V-\partial_s\varphi_{\nu_i\cdot B})\,ds \\
    &+\sum_{i=1}^{k(t)}\left( \int_{\Sigma_i(t)}\vartheta\,ds \right)\fint_{\Sigma_i(t)}\nu_i\cdot B\,ds
  \end{align*}
  The second sum is estimated by Lemma~\ref{lem:estimate-nu-dot-B} below.
  The first sum reads, for any $\eps>0$,
  \begin{equation}
    \label{eq:bulk-estimate-2}
    \begin{split}
      &\,\,-\int_{\Sigma(t)}(\partial_s\vartheta)(\partial_s\varphi_V-\partial_s\varphi_{\nu\cdot B})\,ds \\
      \leq
      &\,\,\frac{1}{2\eps}\int_{\Sigma(t)}(\partial_s\vartheta)^2\,ds
        +\frac{\eps}{2}\int_{\Sigma(t)}\left( \partial_s\varphi_V-\partial_s\varphi_{\nu\cdot B} \right)^2\,ds \\
      \leq
      &\,\,\frac{1}{2\eps}\int_{\Sigma(t)}(\partial_s\vartheta)^2\,ds
        +\eps\int_{\Sigma(t)}\left( \partial_s\varphi_V-\partial_s(\nabla\cdot\xi) \right)^2\,ds
        +\eps\int_{\Sigma(t)}\left( \partial_s(\nabla\cdot\xi)-\partial_s\varphi_{\nu\cdot B} \right)^2\,ds.
    \end{split}
  \end{equation}
  The first term reads
  \begin{align*}
    \frac{1}{2\eps}\int_{\Sigma(t)}(\partial_s\vartheta)^2\,ds
    &\lesssim\frac{1}{2\eps}\int_{\Sigma(t)}(\tau\cdot\nabla\sdist)^2\,ds
      \lesssim\frac{1}{2\eps}\mathcal{E}(t).
  \end{align*}
  The second term is absorbed in the dissipation term for $\mathcal{E}(t)$ for
  $\eps\leq\frac{1}{2}$.
  The last term was already estimated in~\eqref{eq:remaining-term}.
  This concludes the proof of Theorem~\ref{thm:main-thm}.
\end{proof}

\section{Auxiliary results}
\label{sec:aux-results}

Here we state a few auxiliary results which are crucial for the proof of Theorem~\ref{thm:main-thm}.

First we prove existence of weak solutions to the Poisson equation on rectifiable sets.
For a function $u:\Sigma\rightarrow\mathbb{R}$ we denote the average w.r.t. a measure $\mu$ by
\[\langle u\rangle_\Sigma\coloneqq\fint_{\Sigma}u\,d\mu,\]
and we define the space $H^1_{(0)}(\Sigma)\coloneqq\{u\in H^1(\Sigma):\langle u\rangle_\Sigma=0\}$.

\begin{lemma}\label{lem:existence-sol-poisson-eq}
  Let $\Omega$ be a set of finite perimeter, let $\Sigma=\partial^*\Omega$ and let $u\in H^1_{(0)}(\Sigma)$.
  Then there exists a unique function $\varphi\in H^1_{(0)}(\Sigma)$ such that
  \begin{equation}\label{eq:poisson-eq-weak-formulation}
    \int_\Sigma\nabla_\Sigma\varphi\cdot\nabla_\Sigma\psi\,d\mathcal{H}^{d-1}
    =\int_\Sigma u\psi\,d\mathcal{H}^{d-1}
    \quad\text{for all $\psi\in C_c^\infty(\mathbb{R}^d)$.}
  \end{equation}
  Furthermore, if $d=2$, then $\partial_s^2\varphi$ exists and
  \[\partial_s\varphi=f\quad\text{a.e.}\]
\end{lemma}

\begin{proof}
  We will use the Lax--Milgram lemma to prove this result.
  Define a bilinear form $b:H^1_{(0)}(\Sigma)\times H^1_{(0)}(\Sigma)\rightarrow\mathbb{R}$ by
  \[b(\varphi,\psi)\coloneqq\int_{\Sigma}\nabla_\Sigma\varphi\cdot\nabla_\Sigma\psi\,d\mathcal{H}^{d-1}
    \quad\text{for all $\psi\in C_c^\infty(\mathbb{R}^d)$.}
  \]
  By density of $C_c^\infty(\mathbb{R}^d)$, this uniquely defines $b$.
  Since $H^1_{(0)}(\Sigma)\subset L^2(\Sigma)$ is a closed subspace,
  the space $(H^1_{(0)}(\Sigma),\|\cdot\|_{L^2(\Sigma)})$ equipped with the $L^2$ norm is again a Hilbert space.
  By the Poincar\'{e} inequality and Cauchy--Schwarz we have
  \[\sup_{\|\psi\|_{H^1_{(0)}}=1}b(\varphi,\psi)=\|\nabla_\Sigma\varphi\|_{L^2}^2
    \geq\frac{1}{C}\|\varphi\|_{L^2}^2
\quad\text{for all $\varphi\in H^1_{(0)}(\Sigma)$,}\]
  hence $b$ is coercive.
  Furthermore we define a bounded linear operator $f\in(L^2(\Sigma))'$ by
  \[f(\psi)\coloneqq\int_{\Sigma}u\psi\,d\mathcal{H}^{d-1}.\]
  Now by the Lax--Milgram lemma there exists a unique $\varphi\in H^1_{(0)}(\Sigma)$ such that
  \[b(\varphi,\psi)=f(\psi)\quad\text{for all $\psi\in C_c^\infty(\mathbb{R}^d)$.}\]
  Now let $d=2$.
  We test~\eqref{eq:poisson-eq-weak-formulation} with $\varphi$ and obtain
  \begin{align*}
    \int_\Sigma(\partial_s\varphi)^2\,d\mathcal{H}^1
    &=\int_\Sigma\varphi f\,d\mathcal{H}^1.
  \end{align*}
  Approximating $\varphi$ with smooth functions, we get that $\partial_s^2\varphi$ exists a.e.
  and $\partial_s^2\varphi=f$.
\end{proof}

The following simple lemma shows that the surface area of small bubbles is
controlled by the relative energy.

\begin{lemma}\label{lem:area-bdd-by-relative-energy}
  Let $\Sigma=\partial\Omega\subset\mathbb{R}^d$ be a closed surface and let
  $\xi:\mathbb{R}^d\rightarrow\mathbb{R}^d$ be a Lipschitz vector field
  such that $\|\xi\|_\infty\leq 1$.
  If $\diam(\Sigma)\leq\frac{1}{2\|\nabla\xi\|_\infty}$, then
  \begin{equation}
    \label{eq:area-bdd-by-relative-energy}
    \mathcal{H}^{d-1}(\Sigma)\leq 34\int_\Sigma(1-\xi\cdot\nu)\,d\mathcal{H}^{d-1}.
  \end{equation}
\end{lemma}

\begin{proof}
  Let $\alpha\coloneqq\frac{1}{10}$ and $\delta\coloneqq\frac{3}{10}$.
  We claim that
  \begin{equation}
    \label{eq:normal-cannot-point-in-only-one-direction}
    \mathcal{H}^{d-1}(\{x\in\Sigma:\xi(x)\cdot\nu(x)\leq 1-\delta\})
    \geq\alpha\mathcal{H}^{d-1}(\Sigma).
  \end{equation}
  If the claim holds, then
  \begin{align*}
    \mathcal{H}^{d-1}(\Sigma)
    &\leq\frac{1}{\alpha\delta}\int_{\Sigma\cap\{\xi\cdot\nu\leq 1-\delta\}}
      \delta\,d\mathcal{H}^{d-1} \\
    &\leq\frac{1}{\alpha\delta}\int_{\Sigma\cap\{\xi\cdot\nu\leq 1-\delta\}}
      1-\xi\cdot\nu\,d\mathcal{H}^{d-1} \\
    &\leq\int_{\Sigma}
      1-\xi\cdot\nu\,d\mathcal{H}^{d-1}.
  \end{align*}
  Since $\frac{1}{\alpha\delta}=\frac{100}{3}\leq\frac{102}{3}=34$, this implies
  the assertion.

  Now we prove the claim.
  Suppose for a contradiction that
  \begin{equation}
    \label{eq:contradiction-assumption}
    \mathcal{H}^{d-1}(\{x\in\Sigma:\xi(x)\cdot\nu(x)>1-\delta\})
    >(1-\alpha)\mathcal{H}^{d-1}(\Sigma).
  \end{equation}
  Then we have for any $\overline{\xi}\in\mathbb{R}^d$
  \begin{align*}
    \int_{\Sigma}\xi\cdot\nu\,d\mathcal{H}^{d-1}
    &=\int_\Sigma(\xi-\overline{\xi})\cdot\nu\,d\mathcal{H}^{d-1}
  \end{align*}
  and therefore
  \begin{align*}
    \int_\Sigma\xi\cdot\nu\,d\mathcal{H}^{d-1}
    &\leq\|\nabla\xi\|_\infty\diam(\Sigma)\,\mathcal{H}^{d-1}(\Sigma).
  \end{align*}
  On the other hand
  \begin{align*}
    \int_\Sigma\xi\cdot\nu\,d\mathcal{H}^{d-1}
    =\,&\int_{\Sigma\cap\{\xi\cdot\nu>1-\delta\}}\xi\cdot\nu\,d\mathcal{H}^{d-1}
      +\int_{\Sigma\cap\{\xi\cdot\nu\leq 1-\delta\}}\xi\cdot\nu\,d\mathcal{H}^{d-1} \\
    \geq\,&(1-\delta)\mathcal{H}^{d-1}(\{x\in\Sigma:\xi(x)\cdot\nu(x)>1-\delta\}) \\
       &-\mathcal{H}^{d-1}(\{x\in\Sigma:\xi(x)\cdot\nu(x)\leq 1-\delta\}) \\
    \geq\,&((1-\delta)(1-\alpha)-\alpha)
      \mathcal{H}^{d-1}(\Sigma).
  \end{align*}
  By our choice of $\alpha$ and $\delta$, we have
  \begin{align*}
    1-2\alpha-\delta+\alpha\delta
    &\geq 1-\frac{2}{10}-\frac{3}{10}+\frac{3}{100}
      =\frac{1}{2}+\frac{3}{100}
      >\frac{1}{2}.
  \end{align*}
  Hence we get
  \begin{align*}
    \frac{1}{2}\mathcal{H}^{d-1}(\Sigma)
    &<\|\nabla\xi\|_\infty\diam(\Sigma)\mathcal{H}^{d-1}(\Sigma),
  \end{align*}
  a contradiction to~\eqref{eq:contradiction-assumption}.
  This concludes the proof.
\end{proof}

The following construction appeared in~\cite{Laux2022} and will be useful for
our framework as well.

\begin{lemma}\label{lem:existence-B}
  Let $\Sigma^*\subset\mathbb{R}^d$ be a smooth closed surface with enclosed
  region $\Omega^*$ and let $V^*\in C^{2,\alpha}(\Sigma^*)$, for some for some
  $\alpha\in(0,1)$, with $\int_{\Sigma^*} V^*\,d\mathcal{H}^{d-1}=0$.
  Then there exists a compactly supported Lipschitz vector field
  $B\in C_c^{1,1}(\mathbb{R}^d;\mathbb{R}^d)$
  such that $\nu^*\cdot B=V^*$ on $\Sigma^*$ and
  $\nabla\cdot B=O(\dist(\cdot,\Omega^*))$.
\end{lemma}

Applying the lemma to each time slice, we immediately get the following
statement.

\begin{corollary}\label{cor:existence-B}
  There exists a family of compactly supported Lipschitz vector fields
  $B(\cdot,t)\in C_c^{1,1}(\mathbb{R}^d;\mathbb{R}^d)$
  such that $\nu^*\cdot B=V^*$ on $\Sigma^*(t)$ so that
  $\nabla\cdot B(\cdot,t)=O(\dist(\cdot,\Omega^*(t)))$ for all $t$ with
  $\Omega^*(t)$ is as in Definition~\ref{def:strong}.
\end{corollary}

\begin{proof}[Proof of Lemma~\ref{lem:existence-B}]
  The proof consists of three steps.
  First we show existence of a $C^{2,\alpha}$ solution to a Neumann--Laplace
  problem with boundary condition $\nu^*\cdot\nabla\phi=V^*$.
  Second, we improve the regularity to $C^{3,\alpha}$.
  Finally, we use an extension theorem to extend this function to all of $\mathbb{R}^d$.
  \begin{step}
    For each $t$ let $\phi$ be a solution to the Neumann--Laplace problem
    \begin{eqnarray}
      \label{eq:pde-B-1}
      \Delta\phi=&0\qquad&\text{in $\Omega^*$,} \\
      \label{eq:pde-B-2}
      \nu^*\cdot\nabla\phi=&V\qquad&\text{on $\Sigma^*$.}
    \end{eqnarray}
    The existence of a solution $\phi$ with $\int_{\Omega^*}\phi\,dx=0$ follows
    from elliptic theory since
    \begin{align*}
      \int_{\Sigma^*}V^*\,d\mathcal{H}^{d-1}=0.
    \end{align*}
    Using Schauder boundary regularity, see~\cite[Thm.\ 4.1]{Nardi2014}
    or~\cite[Thm.\ 95]{Leoni2013},
    we have, for all $0<\alpha<1$,
    \begin{equation}
      \label{eq:schauder-estimate-phi}
      \|\phi\|_{C^{2,\alpha}\left(\overline{\Omega^*}\right)}
      \leq C(\Omega^*)\|V^*\|_{C^{1,\alpha}\left(\partial\Omega^*\right)}
      \leq C(\Sigma^*).
    \end{equation}
    In the next step, we improve this regularity to show that
    $\phi\in C^{3,\alpha}(\overline{\Omega^*})$, and hence $\nabla\phi\in C^{1,1}(\overline{\Omega^*};\mathbb{R}^d)$.
  \end{step}
  \begin{step}
    It is sufficient to show that for any orthogonal coordinate frame
    $(x^1,\dots,x^d)$ of class $C^{2,\alpha}(\overline{\Omega^*})$ with
    $x^i\cdot\nu=0$ for $1\leq i\leq n-1$ and $x^d=\nu$ on $\partial\Omega^*$.
    Then it holds
    $x^i\cdot\nabla\phi\in C^{2,\alpha}(\overline{\Omega^*})$ for $i=1,\dots,n$.
    Indeed, given such a frame, we see that, for any $i=1,\dots,n-1$,
    $\psi\coloneqq x^i\cdot\nabla\phi$ is a solution to
    \begin{align*}
      \Delta\psi
      &=\Delta x^i\cdot\nabla\phi
        +2\nabla x^i:\nabla^2\phi
      &\text{in $\Omega^*$,} \\
      \nu^*\cdot\nabla\psi
      &=\nabla\phi\cdot(\nu\cdot\nabla)x^i
        +(x^i\cdot\nabla)V^*-\nabla\phi\cdot(x^i\cdot\nabla)\nu^*
      &\text{on $\Sigma^*$,}
    \end{align*}
    where we used that $\Delta\phi=0$ in $\Omega^*$ in the first line and
    $\nu^*\cdot\nabla\phi=0$ on $\Sigma^*$ in the second line.
    We observe that the right-hand side of this PDE is of class
    $C^{1,\alpha}(\overline{\Omega^*})$ and the Neumann boundary datum is of
    class $C^{1,\alpha}(\partial\Omega^*)$,
    because $V^*\in C^{2,\alpha}(\partial\Omega^*)$ and
    $\partial\Omega^*\in C^{3,\alpha}$, which implies
    $(x^i\cdot\nabla)\nu^*\in C^{1,\alpha}(\partial\Omega^*)$ for all $i=1,\dots,d-1$.
    For the normal field $x^d$, we observe that
    $\psi\coloneqq x^d\cdot\nabla\phi$ solves the Dirichlet--Laplace problem
    \begin{align*}
      \Delta\psi
      &=\Delta x^d\cdot\nabla\phi+2\nabla x^d:\nabla^2\phi
      &\text{in $\Omega^*$,} \\
      \psi
      &=\nu^*\cdot\nabla\phi
        =V^*
      &\text{on $\Sigma^*$.}
    \end{align*}
    The right-hand side of this PDE is identical to the previous case with the
    Dirichlet datum being of class $C^{2,\alpha}(\partial\Omega^*)$.
    Thus we can apply Schauder boundary regularity for the Dirichlet--Laplace
    problem~\cite[Thm.\ 6.8]{Gilbarg2001} to assert that also
    $\psi\in C^{2,\alpha}(\overline{\Omega^*})$.
    Therefore we have
    \begin{align*}
      \|\phi\|_{C^{3,\alpha}\left( \overline{\Omega^*} \right)}\leq C(\Sigma^*).
    \end{align*}
  \end{step}
  \begin{step}
    We extend $\phi$ to all of $\mathbb{R}^d$ using a standard extension theorem,
    e.g.~\cite[Thm.\ 6.37]{Gilbarg2001}, to a function
    $\overline\phi\in C^{3,\alpha}(\mathbb{R}^d)$ with the same regularity as
    $\phi$ such that $\overline\phi=\phi$ in
    $\overline{\Omega^*}$ and $\overline\phi=0$
    in $\mathbb{R}^d\setminus B_R(0)$ for some $R<\infty$ sufficiently large.
    Now we set
    \begin{align*}
      B\coloneqq\nabla\overline\phi.
    \end{align*}
    Finally we have $\nabla\cdot B=\Delta\overline\phi=0$ on $\Sigma^*$ and therefore
    $\nabla\cdot B=O(\dist(\cdot,\Omega^*))$.
    \qedhere
  \end{step}
\end{proof}

The combination of the simple geometric estimate in Lemma~\ref{lem:area-bdd-by-relative-energy} and the properties of the extended velocity vector field in Corollary~\ref{cor:existence-B} gives us the following crucial estimate.

\begin{lemma}\label{lem:estimate-nu-dot-B}
  Let $\Sigma(t)$ be as in Definition~\ref{def:weak} with connected components
  $\Sigma_i(t)$, where $1\leq i\leq k(t)$, and $\Sigma^*(t)$ as in
  Definition~\ref{def:strong}.
  Let $B(\cdot,t)\in C_c^{1,1}(\mathbb{R}^d;\mathbb{R}^d)$ as in
  Corollary~\ref{cor:existence-B}.
  Then there exists a constant $C$ depending only on
  $\Sigma^*$ such that
  \begin{equation}
    \label{eq:estimate-nu-dot-B}
    \sum_{i=1}^{k(t)}\left| \int_{\Sigma_i(t)}\nu_i\cdot B\,d\mathcal{H}^{d-1} \right|
    \leq C\mathcal{F}(t)
  \end{equation}
  and, for $d=2$,
  \begin{equation}
    \label{eq:estimate-fint-nu-dot-B}
    \sum_{i=1}^{k(t)}\frac{1}{\mathcal{H}^{d-1}(\Sigma_i(t))}
    \left| \int_{\Sigma_i(t)}\nu_i\cdot B\,d\mathcal{H}^{d-1} \right|
    \leq C(\mathcal{E}(t)+\mathcal{F}(t)).
  \end{equation}
\end{lemma}

Again we will suppress the $t$-dependence in the following proof.

\begin{proof}
  Let $R<\infty$ be sufficiently large such that
  $\supp\nabla B\subset B_R(0)\times[0,T^*]$.
  We compute~\eqref{eq:estimate-nu-dot-B}:
  Since
  $|\nabla\cdot B|=O(\dist(\cdot,\Omega^*(t)))\leq C|\sdist|(1-\chi_{\Omega^*(t)})$,
  \begin{align*}
    \sum_{i=1}^{k(t)}\left| \int_{\Sigma_i}\nu_i\cdot B\,ds \right|
    &=\sum_{i=1}^{k(t)}\left| \int_{\Omega_i}\nabla\cdot B\,dx \right| \\
    &\lesssim\int_{B_R(0)}|\sdist|(1-\chi^*)\chi\,dx \\
    &\leq\frac{2R}{\delta}\int_{\mathbb{R}^d}|\vartheta||\chi-\chi^*|\,dx \\
    &\leq\frac{2R}{\delta}\|\nabla\cdot B\|_\infty\mathcal{F}.
  \end{align*}
  Finally, in the case of $d=2$, we let $C\coloneqq\|B\|_\infty$.
  We split the sum~\eqref{eq:estimate-fint-nu-dot-B} into two parts:
  \begin{align*}
    \sum_{i=1}^{k(t)}\frac{1}{\mathcal{H}^{1}(\Sigma_i)}
    \left| \int_{\Sigma_i}\nu_i\cdot B\,d\mathcal{H}^{1} \right|
    \leq\,
    &\sum_{i:\mathcal{H}^{1}(\Sigma_i)>1/C}\frac{1}{\mathcal{H}^{1}(\Sigma_i)}
      \left| \int_{\Sigma_i}\nu_i\cdot B\,d\mathcal{H}^{1} \right| \\
    &+\sum_{i:\mathcal{H}^{1}(\Sigma_i)\leq 1/C}\frac{1}{\mathcal{H}^{1}(\Sigma_i)}
      \left| \int_{\Sigma_i}\nu_i\cdot B\,d\mathcal{H}^{1} \right|.
  \end{align*}
  For the second sum we use the isoperimetric inequality $|\Omega_i|\lesssim(\mathcal{H}^{1}(\Sigma_i))^2$
  and then, by the elementary estimate
  $\diam(\Sigma_i)\leq\frac{1}{2}\mathcal{H}^1(\Sigma_i)$,
  we may use Lemma~\ref{lem:area-bdd-by-relative-energy} to obtain
  \begin{align*}
    \sum_{i:\mathcal{H}^{1}(\Sigma_i)\leq 1/C}\frac{1}{\mathcal{H}^{1}(\Sigma_i)}
    \left| \int_{\Sigma_i}\nu_i\cdot B\,d\mathcal{H}^{1} \right|
    &=\sum_{i:\mathcal{H}^{1}(\Sigma_i)\leq 1/C}\frac{1}{\mathcal{H}^{1}(\Sigma_i)}
      \left| \int_{\Omega_i}\nabla\cdot B\,dx \right| \\
    &\lesssim\|\nabla\cdot B\|_\infty\sum_{i:\mathcal{H}^{1}(\Sigma_i)\leq 1/C}\mathcal{H}^{1}(\Sigma_i) \\
    &\leq 34\|\nabla\cdot B\|_\infty\sum_{i=1}^{k}\int_{\Sigma_i}1-\xi\cdot\nu_i\,d\mathcal{H}^{1} \\
    &=34\|\nabla\cdot B\|_\infty\mathcal{E}.
  \end{align*}
  For the first sum we can use the estimate~\eqref{eq:estimate-nu-dot-B}.
\end{proof}

\begin{lemma}
  Let $\Omega\subset\mathbb{R}^2$ be a set of finite perimeter, $\Sigma=\partial^*\Omega$
  and let $f\in C^{0,1}(\Sigma)$.
  Then
  \[\int_\Sigma\partial_s f\,d\mathcal{H}^1=0.\]
\end{lemma}

\begin{proof}
  Let $f\in C_c^\infty(\mathbb{R}^d)$ and let $\nabla^\perp f\coloneqq J^{-1}\nabla f$.
  Then
  \begin{align*}
    \int_\Sigma\partial_s f\,d\mathcal{H}^1
    &=\int_\Sigma\tau\cdot\nabla f\,d\mathcal{H}^1 \\
    &=\int_\Sigma\nu\cdot\nabla^\perp f\,d\mathcal{H}^1 \\
    &=\int_\Omega(-\partial_{x_1}\partial_{x_2}f+\partial_{x_2}\partial_{x_1}f)\,dx \\
    &=0.
  \end{align*}
  By density of $C_c^\infty(\mathbb{R}^2)$ in $C^{0,1}(\Sigma)$, this concludes the proof.
\end{proof}

\subsection{Jordan boundary decomposition}

We collect some crucial properties of sets of finite perimeter in the plane and a corresponding version of connectedness,
called \emph{indecomposability}.
More details can be found in~\cite{Ambrosio2001}.

\begin{definition}[Indecomposable sets]
  Let $\Omega\subset\mathbb{R}^d$ be a set of finite perimeter.
  We say that $\Omega$ is \emph{decomposable} if there exists a disjoint partition
  $\Omega=A\cup B$ such that $P(\Omega)=P(A)+P(B)$.
  If $\Omega$ is not decomposable, we call $\Omega$ \emph{indecomposable}.
\end{definition}

\begin{definition}[Simple sets and Jordan boundaries]\label{def:jordan-boundaries}
  An indecomposable set $\Omega\subset\mathbb{R}^d$ is called simple
  if for all sets of finite perimeter $F\subset\mathbb{R}^d$ the following holds:
  If $|F|\in(0,\infty)$ and $\partial^* F\subset\partial^*\Omega$ up to $\mathcal{H}^{d-1}$ null sets,
  then $E=F$.
  
  A set $J\subset\mathbb{R}^d$ is called a \emph{Jordan boundary} if there exists a simple set $\Omega$
  such that $\partial^*\Omega=J$.
\end{definition}

\begin{proposition}[Decomposition theorem~{\cite[Thm.\ 1]{Ambrosio2001}}]\label{prop:decomposition-theorem}
  Let $\Omega$ be a set of finite perimeter.
  Then there exists a unique finite or countable family of pairwise disjoint indecomposable sets $\{\Omega_i\}_{i}$
  such that $|\Omega_i|>0$ for all $i$ and $P(\Omega)=\sum_{i}P(\Omega_i)$.
  Furthermore any indecomposable set $F\subset\Omega$ is contained in some $\Omega_i$ up to Lebesgue null sets.
\end{proposition}

To simplify the following statement, we add the formal Jordan boundary $J_o$ whose interior is empty
and $J_\infty$ whose interior is $\mathbb{R}^d$.

\begin{proposition}[Jordan boundary decomposition~{\cite[Thm.\ 4]{Ambrosio2001}}]\label{prop:jordan-boundaries}
  Let $\Omega$ be a set of finite perimeter.
  Then there exists a unique decomposition of $\partial^*\Omega$ into
  Jordan boundaries $\{J_i^+,J_k^-:i,k\in\mathbb{N}\}$
  such that the following hold:
  \begin{enumerate}
    \item Given $\interior(J_i^+),\interior(J_j^+)$, then either they are disjoint or one is contained in the other;
      Given $\interior(J_k^-),\interior(J_l^-)$, then either they are disjoint or one is contained in the other.
    \item Every $\interior(J_k^-)$ is contained in some $\interior(J_i^+)$.
    \item $P(\Omega)=\sum_{i\in\mathbb{N}}\mathcal{H}^{d-1}(J_i^+)+\sum_{j\in\mathbb{N}}\mathcal{H}^{d-1}(J_k^-)$.
    \item If $\interior(J_i^+)\subseteq\interior(J_j^+)$ for some $i\neq j$, then there exists some $J_k^-$
      such that $\interior(J_k^-)\subseteq\interior(J_l^-)\subseteq\interior(J_j^+)$.
      Similarly, if $\interior(J_k^-)\subseteq\interior(J_l^-)$ for some $k\neq l$, then there exists some $J_i^+$
      such that $\interior(J_k^-)\subseteq\interior(J_i^+)\subseteq\interior(J_l^-)$.
    \item Let $L_i\coloneqq\{k:\interior(J_k^-)\subset\interior(J_i^+)\}$ and let
      $Y_i\coloneqq\interior(J_i^+)\setminus\left( \bigcup_{k\in L_i}\interior(J_k^-) \right)$.
      Then $Y_i$ are pairwise disjoint, indecomposable and $\Omega=\bigcup_{i\in\mathbb{N}}Y_i$.
  \end{enumerate}
\end{proposition}

\begin{remark}
  The authors of~\cite{Ambrosio2001} define $\partial^*\Omega$ as the set of points where density
  \[D(x)\coloneqq\lim_{r\downarrow 0}\frac{|\Omega\cap B(x,r)|}{|B(x,r)|}\]
  is equal to $1/2$, denoted by $\Omega^{1/2}\coloneqq\{x:D(x)=1/2\}$.
  However, up to $\mathcal{H}^{d-1}$ null sets we have
  \[|\nabla\chi_\Omega|=\mathcal{H}^{d-1}\MNSlefthalfcup\Omega^{1/2},\]
  thus the results are still applicable in our setting. 
\end{remark}

\begin{remark}
  The sets $\Omega_i$ in Proposition~\ref{prop:decomposition-theorem} 
  are exactly the $Y_i$ from Proposition~\ref{prop:jordan-boundaries}.
  We also recall De Giorgi's structure theorem:
  For any finite perimeter set $\Omega\subset\mathbb{R}^d$ there exists a null set $N$ 
  and countably many $C^1$ hypersurfaces $S_i$ and compact subsets $K_i\subset S_i$ such that
  \[\partial^*\Omega=N\cup\bigcup_{i=1}^\infty K_i.\] 
\end{remark}

\begin{lemma}\label{lem:integral-curvature-of-jordan-boundary}
  Let $d=2$ and let $\Omega$ be a set of finite perimeter, $\Sigma=\partial^*\Omega$,
  let $\kappa$ denote the mean curvature of $\Sigma$ and let $J_i$ denote
  its Jordan boundary decomposition as in Proposition~\ref{prop:jordan-boundaries}.
  Then
  \[\int_{J_i}\kappa|_{J_i}\,d\mathcal{H}^1=2\pi\quad\text{for all $i$.}\]
\end{lemma}

\begin{proof}
  By definition we have
  \[\int_{\Sigma}\div_{\Sigma}B\,d\mathcal{H}^{1}
    =\int_{\Sigma}\kappa\nu\cdot B\,d\mathcal{H}^{1}
  \]
  for all $B\in C_c^1(\mathbb{R}^2;\mathbb{R}^2)$.
  Now fix $i$ and let $B\in C_c^1(\mathbb{R}^2;\mathbb{R}^2)$.
  Then, by the definition of the Jordan boundary decomposition,
  we can take a cutoff function $\eta$ such that $\tilde{B}\coloneqq\eta B=B$ on $J_i$
  and $\supp\tilde{B}\subset\mathbb{R}^2\setminus\bigcup_{j\neq i}J_j$.
  Then
  \begin{align*}
    \int_{J_i}\kappa|_{J_i}\nu\cdot B
    &=\int_{\Sigma}\kappa\nu\cdot\tilde{B}
    =\int_{\Sigma}\div_{\Sigma}\tilde{B}
    =\int_{J_i}\div_{\Sigma}B.
  \end{align*}
  Hence $\kappa|_{J_i}$ is the mean curvature of $J_i$.
  We now show that $J_i$ is a rectifiable closed curve, then the claim follows from the Gauss--Bonnet theorem.
  Indeed, let $E$ be a simple set such that $\partial^* E=J_i$ and let $Y\coloneqq\interior J_i$.
  Then $\partial^* Y$ is either a rectifiable closed curve or a union of closed curves
  and $\partial^* Y\subset\partial^* E$.
  If $\partial^* Y$ is a union of curves, then there exist curves with non-empty interior $I_1,I_2$
  such that $Y=\interior(I_1)\cup\interior(I_2)$.
  However, since $E$ is simple, we have $Y=E=\interior(I_1)\cup\interior(I_2)$.
  But this means that $E$ is decomposable, a contradiction.
\end{proof}

\subsection{A Poincar\'{e} inequality on rectifiable sets}

We introduce a Poincar\'{e} inequality on rectifiable sets 
by viewing them as metric measure spaces, cf.~\cite{Heinonen2001}.
This in turn gives us a Poincar\'{e} inequality on the reduced boundary of finite perimeter sets.

\begin{definition}
  Let $\Sigma\subset\mathbb{R}^d$ be a rectifiable set and let $x,y\in\Sigma$.
  A path connecting $x$ and $y$ is a continuous piecewise $C^1$ function $\gamma_{xy}:[0,L]\rightarrow\Sigma$ 
  such that $\gamma_{xy}(0)=x$ and $\gamma_{xy}(L)=y$.
  The length of a path $\gamma$ is denoted by $L(\gamma)\coloneqq\int_0^L|\gamma'(s)|\,ds$.
  The distance between $x$ and $y$ is defined as
  \begin{equation}\label{eq:path-dist}
    d_\Sigma(x,y)\coloneqq\inf_{\gamma_{xy}}L(\gamma_{xy}),
  \end{equation}
  where we take the infimum over all paths $\gamma_{xy}$ connecting $x$ and $y$.
\end{definition}

To obtain a Poincar\'{e} inequality on $\Sigma$, we define the space $M^{1,p}(\Sigma)$ by
\[M^{1,p}(\Sigma)\coloneqq\{u\in L^p(\Sigma):\exists g_u\in L^p(\Sigma)\},\]
where $g_u$ is a function satisfying
\begin{equation}\label{eq:grad-M}
  |u(x)-u(y)|\leq d_\Sigma(x,y)(g_u(x)+g_u(y))\quad\text{for a.e.\ $x,y\in\Sigma$.}
\end{equation}

\begin{remark}
  If we equip $M^{1,p}(\Sigma)$ with a norm defined by
  \[\|u\|_{M^{1,p}}^p\coloneqq\|u\|_{L^p}^p+\inf_{g_u}\|g_u\|_{L^p}^p,\]
  then $M^{1,p}(\Sigma)$ becomes a Banach space.
\end{remark}

\begin{definition}
  Let $\Sigma$ be a rectifiable set.
  We say that a measure $\mu$ on $\Sigma$ is doubling if there exists a constant $C(\mu)$ such that
  \[\mu(2B)\leq C(\mu)\mu(B)\quad\text{for all balls $B\subset\Sigma$.}\]
  The constant $C(\mu)$ is called the doubling constant of $\mu$.
\end{definition}

Here the balls are taken w.r.t.\ the distance~\eqref{eq:path-dist}.
Note that for $\mu=\mathcal{H}^{d-1}$ we can choose $C(\mu)=2^{d-1}$.

For the remainder of this section, $\Sigma$ is a rectifiable set and $\mu$ is a doubling measure on $\Sigma$,
unless otherwise noted.
Now the following Poincar\'{e} inequality holds, which is a corollary of~\cite[Thm.\ 5.15]{Heinonen2001}
and holds in particular for the reduced boundary of finite perimeter sets equipped with the Gauss--Green measure.

\begin{proposition}\label{prop:rect-set-poincare-ineq}
  Let $\Sigma$ be a rectifiable set, $1\leq p<\infty$, and let $\mu$ be a doubling measure on $\Sigma$.
  Then there exists a constant $C$ depending only on $p$ and the doubling constant of $\mu$ 
  such that for all $u\in M^{1,p}(\Sigma)$
  and every $g$ satisfying~\eqref{eq:grad-M} we have
  \begin{equation}
    \int_{\Sigma}|u-\langle u\rangle_{\Sigma}|^p\,d\mu
    \leq C 2^p(\diam\Sigma)^p\int_{\Sigma}|g|^p\,d\mu.
  \end{equation}
\end{proposition}

\begin{lemma}\label{lem:cont-func-in-M-1p}
  Let $\Sigma\subset\mathbb{R}^2$ be a rectifiable closed curve with $\mathcal{H}^1(\Sigma)<\infty$.
  Suppose further that there exists a constant $C<\infty$ such that
  $\Sigma$ satisfies~\eqref{eq:bound-balls}.
  Then there exists a constant $C'<\infty$ depending only on $p$ such that
  $C^{0,1}(\Sigma)\subseteq M^{1,p}(\Sigma)$ for all $p\geq 1$ and
  \begin{align*}
    \int_\Sigma|u-\langle u\rangle_\Sigma|^p\,d\mathcal{H}^1
    &\leq C'2^{p}(\diam\Sigma)^p\int_\Sigma|\partial_s u|^p\,d\mathcal{H}^1.
  \end{align*}
\end{lemma}

This lemma also holds in higher dimensions with an analogous proof, but we restrict ourselves to the 2D case for simplicity.
Before we continue with the proof, we recall some facts about maximal functions.
For a locally integrable function $f:\Sigma\rightarrow\mathbb{R}$ we define the maximal function $M(f)$ by
\begin{align*}
  M(f)(x)\coloneqq\sup_{r>0}\fint_{B^\Sigma(x,r)}|f|\,d\mu.
\end{align*}
Here $B^\Sigma(x,r)$ denotes the ball with radius $r$ centered at $x\in\Sigma$
with respect to the intrinsic distance~\eqref{eq:path-dist}.
Further $\partial_s$ will denote differentiation with respect to the arc-lenght parameter 
on $\Sigma$, as above.

\begin{proof}[Proof of Lemma~\ref{lem:cont-func-in-M-1p}]
  Let $u\in C^{0,1}(\Sigma)$ be a Lipschitz function.
  By Rademacher's theorem $u$ is differentiable a.e.\ with derivative $\partial_s u$.
  Let $M(\partial_s u)$ denote the maximal function of $\partial_s u$, i.e.
  \[M(\partial_s u)(x)\coloneqq\sup_{r>0}\fint_{B^\Sigma(x,r)}|\partial_s u|\,d\mathcal{H}^{1}.\]
  Furthermore observe that, since $\partial_s u$ is locally Lipschitz,
  we have $M(\partial_s u)(x)\geq|\partial_s u(x)|$ for a.e.\ $x\in\Sigma$.
  Now, let $x\in\Sigma$ and let $B^\Sigma\subset\Sigma$ be any (intrinsic) ball containing $x$.
  Further, for $y\in\Sigma$, let $\gamma_{xy}$ be a shortest path connecting $x$ and $y$ which is parametrized by arc length.
  Such a path exists since $\Sigma$ is compact.
  Then
  \begin{align*}
    \mathcal{H}^1\left(B^\Sigma\right)(u(x)-\langle u\rangle_{B^\Sigma})
    &=\int_{B^\Sigma}\int_{0}^{d(x,y)}\partial_s u(\gamma_{xy}(r))\,dr\,d\mathcal{H}^1(y) \\
    &\leq\int_{B^\Sigma}|B^\Sigma(x,d(x,y))|M(\partial_s u)(x)d\mathcal{H}^1(y) \\
    &\leq\mathcal{H}^1\left(B^\Sigma\right)^2 M(\partial_s u)(x).
  \end{align*}
  Now, for $x,y\in\Sigma$ we can take a ball $B$ with radius $d(x,y)$ such that $x,y\in B$ and obtain, by symmetry,
  \begin{align*}
    |u(x)-u(y)|
    &\leq\mathcal{H}^1\left(B^\Sigma\right)(M(\partial_s u)(x)+M(\partial_s u)(y)).
  \end{align*}
  Now, since there exists a constant $\tilde{C}$ such that
  $\mathcal{H}^1\left(B^\Sigma\right)\leq\tilde{C}r$ for any ball of radius $r$ in $\Sigma$,
  we get that $\tilde{C}M(\partial_s u)$ satisfies~\eqref{eq:grad-M}.
  Finally, using the maximal function theorem~\cite[Thm.\ 2.2]{Heinonen2001},
  there exists a constant $C$ depending only on the doubling constant of $\mathcal{H}^1\MNSlefthalfcup\Sigma$ such that
  \[\int_\Sigma M(\partial_s u)^p\,d\mathcal{H}^1
    \leq C\int_\Sigma|\partial_s u|^p\,d\mathcal{H}^1,
  \]
  and by Proposition~\ref{prop:rect-set-poincare-ineq} we have
  \begin{align*}
    \int_\Sigma|u-\langle u\rangle_\Sigma|^p\,d\mathcal{H}^1
    &\leq C2^{p}(\diam\Sigma)^p\int_\Sigma M(\partial_s u)^p\,d\mathcal{H}^1 \\
    &\leq C\tilde{C}2^{p}(\diam\Sigma)^p\int_\Sigma|\partial_s u|^p\,d\mathcal{H}^1.\qedhere
  \end{align*}
\end{proof}

The proof can easily be generalized to arbitrary dimensions, cf.~\cite[Chapter 2]{Heinonen2001}.
Note that in the case when $\Sigma(t)=\partial^*\Omega(t)$ is a weak solution to surface diffusion
in the sense of Definition~\ref{def:weak} and $\mu=\mathcal{H}^{d-1}\MNSlefthalfcup\Sigma(t)$,
then $\mu$ has doubling constant $2^{d-1}$ and hence on each path component $\Sigma_i(t)$ we have
\begin{equation}\label{eq:poincare-ineq}
  \int_{\Sigma_i(t)}|u-\langle u\rangle_{\Sigma_i(t)}|^p\,d\mu
  \leq 2^{p+d-1}C\tilde{C}(\mathcal{H}^{d-1}(\Sigma(0)))^p\int_{\Sigma_i(t)}|\nabla_{\Sigma_i(t)}u|^p\,d\mu
\end{equation}
for all $u\in C^{0,1}(\Sigma(t))$ and a.e.\ $t$.
In particular the Poincar\'{e} constant can be bounded by a constant depending only on $p, d$ and $\Sigma(0)$.

\section*{Acknowledgments}
This project has received funding from the Deutsche Forschungsgemeinschaft (DFG, German Research Foundation) under Germany's Excellence Strategy -- EXC-2047/1 -- 390685813.

\bibliographystyle{abbrv}
\bibliography{bibfile.bib}

\end{document}